\documentclass[11pt,a4paper,twoside]{article}      

\usepackage{graphicx}
\usepackage[T1]{fontenc}
\usepackage[cp1250]{inputenc} 
\usepackage{times}

\usepackage[leqno]{amsmath} 
\usepackage{amsfonts} 
\usepackage{amsthm}
\usepackage{bm} 

\input xy  \xyoption{all}

\usepackage{geometry}
\usepackage[usenames]{color} 

\usepackage{psfrag}

\usepackage{enumerate} 
\usepackage{hyperref} 

\graphicspath{} 

\newcommand{\new}{}

\newcommand{\rzut}{\operatorname{proj}}
\newcommand{\R}{\mathbb{R}} 
\newcommand{\RR}{\mathcal{R}} 
\newcommand{\Sec}{\Gamma} 
\newcommand{\PP}{\mathbb{P}} 
\newcommand{\proj}[1]{\mathbb{P}(#1)} 
\newcommand{\M}{\mathcal{M}} 
\newcommand{\C}{\mathcal{C}} 
\newcommand{\D}{\mathcal{D}} 
\newcommand{\B}{\mathcal{B}} 
\newcommand{\N}{\mathrm{N}} 
\newcommand{\X}{\mathfrak{X}} 
\newcommand{\rank}{\operatorname{rank}} 
\newcommand{\HH}{\mathcal{H}} 
\newcommand{\ul}{\underline}
\newcommand{\K}{\mathcal{K}} 
\newcommand{\sym}[1]{C_{#1}} 
\newcommand{\cont}[1]{\vec{C}_{#1}} 
\newcommand{\U}{\mathcal{U}} 

\newcommand{\A}[1]{A_{#1}}
\newcommand{\flow}[1]{F_{#1}}
\newcommand{\Flow}[1]{\bm{F}_{#1}}

\newcommand{\id}{\operatorname{id}}

\newcommand{\wt}[1]{\bm{#1}}
\newcommand{\wtilde}[1]{\widetilde{#1}}
\newcommand{\wh}[1]{\widehat{#1}}
\newcommand{\T}{\mathrm{T}}
\newcommand{\lra}{\longrightarrow}
\newcommand{\ra}{\rightarrow}

\newcommand{\dd}{\operatorname{d}}
\newcommand{\vect}{\operatorname{vect}_{\R}}
\newcommand{\pa}{\partial}
\newcommand{\ad}{\operatorname{ad}}

\newcommand{\sexy}{charming} 
\newcommand{\Sexy}{Charming} 
\newcommand{\nice}{Caratheodory} 
\def\<#1>{\big\langle #1\big\rangle}
\def\(#1){\left( #1\right)}

\numberwithin{equation}{section} 

\theoremstyle{plain} 
\newtheorem{theorem}{Theorem}[section]
\newtheorem{proposition}[theorem]{Proposition}
\newtheorem{lemma}[theorem]{Lemma}

\theoremstyle{definition}
\newtheorem{definition}[theorem]{Definition}
\newtheorem{example}[theorem]{Example}
\newtheorem{corollary}[theorem]{Corollary}

\theoremstyle{remark}
\newtheorem{remark}[theorem]{Remark}

\begin{document}

\title{A contact covariant approach to optimal control\\
with applications to sub-Riemannian geometry\thanks{This research was supported by the National Science Center under the grant DEC-2011/02/A/ST1/00208 ``Solvability, chaos and control in quantum systems''.}
}

\author{Micha\l{} J\'{o}\'{z}wikowski\footnote{email: mjozwikowski@gmail.com},\\
Center for Theoretical Physics,\\
 Polish Academy of Sciences\\[0.1cm]
and \\[0.1cm]
Institute of Mathematics,\\
 Polish Academy of Sciences\\[0.5cm]
 Witold Respondek\footnote{email: witold.respondek@insa-rouen.fr}\\
Normandie Universit\'{e}, France\\
 INSA de Rouen, Laboratoire de Math\'{e}matiques}

\date{\today}

\maketitle

\begin{abstract}
We discuss contact geometry naturally related with optimal control problems (and Pontryagin Maximum Principle). We explore and expand the observations of Ohsawa \cite{Ohsawa_contact_pmp_2015}, providing simple and elegant characterizations of normal and abnormal sub-Riemannian extremals. 
\end{abstract}



\section{Introduction}

\paragraph{A contact interpretation of the Pontryagin Maximum Principle.}
In a recent paper Ohsawa \cite{Ohsawa_contact_pmp_2015} observed that for normal solutions of the optimal control problem on a manifold $Q$, the Hamiltonian evolution of the covector $\bm \Lambda_t$ in $\T^\ast (Q\times \R)$ considered in the Pontryagin Maximum Principle (PMP, in short), projects to a well-defined contact evolution in the projectivization $\PP(\T^\ast (Q\times \R))$. Here $Q\times \R$ is the extended configuration space (consisting of both the configurations $Q$ and the costs $\R$) and $\PP(\T^\ast (Q\times \R))$ is equipped with a natural contact structure. Moreover, Ohsawa observed that the maximized Hamiltonian of the PMP is precisely the generating function of this contact evolution. 

The above result was our basic inspiration to undertake this study. Our goal was to understand, from a geometric viewpoint, the role and origins of the above-mentioned contact structure in the PMP and to study possible limitations of the contact approach (does it work alike for abnormal solutions, etc.). 

As a result we prove Theorem \ref{thm:pmp_contact}, a version of the PMP, in which the standard Hamiltonian evolution of a covector curve $\wt \Lambda_t$ in $\T^\ast(Q\times\R)$ along an optimal solution $\wt q(t)\in Q\times\R$ is substituted by a contact evolution of a curve of hyperplanes $\wt\HH_t$ in $\T(Q\times\R)$ along this solution. (Note that the space of all hyperplanes in $\T(Q\times\R)$ is actually the manifold of contact elements of $Q\times\R$ and can be naturally identified with $\PP(\T^\ast (Q\times\R))$.) It is worth mentioning that this result is valid regardless of the fact whether the solution is normal or abnormal and, moreover, the contact evolution is given by a natural contact lift of the extremal vector field (regarded as a time-dependent vector field on $Q\times\R$). Finally, using the well-known relation between contact vector fields and smooth functions we were able to interpret the Pontryagin maximized Hamiltonian as a generating function of the contact evolution of $\wt \HH_t$. 

It seems to us that, apart from the very recent paper of Ohsawa \cite{Ohsawa_contact_pmp_2015}, the relation between optimal control and contact geometry has not been explored in the literature. This fact is not difficult to explain as the PMP in its Hamiltonian formulation has been very successful and as symplectic geometry is much better developed and understood than contact geometry. In our opinion, the contact approach to the PMP seems to be a promising direction of studies for at least two reasons. First of all it allows for a unified treatment of normal and abnormal solutions and, secondly, it seems to be closer to the actual geometric meaning of the PMP (we shall justify this statement below).

\paragraph{About the proof.}
The justification of Theorem \ref{thm:pmp_contact} is rather trivial. In fact, it is just a matter of interpretation of the classical proof of the PMP \cite{Pontr_Inn_math_theor_opt_proc_1962} (see also \cite{Lewis_2006} and \cite{Liberzon_2012}). Recall that geometrically the PMP says that at each point of the optimal trajectory $\wt q(t)$, the cone $\wt\K_t\subset \T_{\wt q(t)}(Q\times\R)$ approximating the reachable set can be separated, by a hyperplane $\wt\HH_t\subset\T_{\wt q(t)}(Q\times\R)$, from the direction of the decreasing cost (cf. Figure \ref{fig:pmp}). Thus in its original sense the PMP describes the evolution of a family of hyperplanes $\bm\HH_t$ (i.e., a curve in the manifold of contact elements of $Q\times\R$, identified with $\PP(\T^\ast(Q\times\R))$) along the optimal solution. This evolution is induced by the flow of the optimal control on $Q\times\R$. From this perspective the only ingredient one needs to prove Theorem \ref{thm:pmp_contact} is to show that this flow induces a contact evolution (with respect to the natural contact structure) on $\PP(\T^\ast(Q\times\R))$. It is worth mentioning that the covector curve $\wt\Lambda_t\in\T^\ast(Q\times\R)$ from the standard formulation of the PMP is nothing else than just an alternative description of the above-mentioned curve of hyperplanes, i.e., $\wt\HH_t=\ker\wt\Lambda_t$ for each time $t$. Obviously, there is an ambiguity in choosing such a $\wt\Lambda_t$, which is defined up to a rescaling. 

\paragraph{Applications.}
From the above perspective it is obvious that the description of the necessary conditions for optimality of the PMP in terms of $\wt\HH_t$'s (the contact approach) is closer to the actual geometric meaning of the PMP as it contains the direct information about the separating hyperplanes. On the contrary, in the Hamiltonian approach this information is translated into the language of covectors (not to forget the non-uniqueness of the choice of $\wt\Lambda_t$). 

We illustrate the contact approach to the PMP by discussing its applications to the sub-Riemannian (SR, in short) geodesic problem in Section \ref{sec:appl}. Recall that a SR geodesic problem on a manifold $Q$ is an optimal control problem in which the controls parametrize trajectories tangent to a smooth distribution $\D\subset\T Q$ and the cost of a trajectory is its length calculated via a given positively defined bilinear form $g:\D\times\D\ra\R$ (the SR metric). Actually, due to the Cauchy-Schwartz inequality, the trajectories minimizing the length are exactly those that minimize the kinetic energy and are parametrized by the arc-length. In such a setting, using  some elementary geometric considerations, we were able to relate $\D$ and $g$ with the separating hyperplanes $\wt\HH_t$ (Lemma \ref{prop:dim_H}). In consequence, still using elementary arguments, the following two results about SR extremals were derived:
\begin{itemize}
	\item Theorem \ref{thm:abnormal} completely characterizes abnormal SR extremals. It states that an absolutely continuous curve $q(t)\in Q$ tangent to $\D$ is an abnormal extremal if and only if the minimal distribution along $q(t)$ which contains $\D_{q(t)}$ and is invariant along $q(t)$ under the flow of the extremal vector field is of rank smaller than $\dim Q$. As a special case (for smooth vector fields) we obtain, in Corollary \ref{cor:abnormal_smooth}, the following result: if the distribution spanned by the iterated Lie brackets of a given $\D$-valued vector field $X\in\Sec(\D)$ with all possible $\D$-valued vector fields, i.e.,
		$$\<\ad_X^k(Z)\ |\ Z\in \Sec(\D),\quad k=0,1,2,\hdots>$$
is of constant rank smaller than $\dim Q$, then the integral curves of $X$ are abnormal SR extremals. 

	\item Theorem \ref{thm:normal} in a similar manner (yet under an additional assumptions that the controls are normalized with respect to the SR metric $g$) provides a complete characterization of normal SR extremals. It states that an absolutely continuous curve $q(t)\in Q$, tangent to $\D$, is a normal extremal if and only if it is of class $C^1$ with an absolutely continuous derivative and if the minimal distribution along $q(t)$ which contains these elements of $\D_{q(t)}$ that are $g$-orthogonal to $\dot q(t)$ and is invariant along $q(t)$ under the flow of the extremal vector field does not contain the direction tangent to $q(t)$ at any point. Again in the smooth case we conclude, in Corollary \ref{cor:normal}, that if for a given normalized vector field $X\in\Sec(\D)$ the distribution spanned by the iterated Lie brackets of $X$ with all possible $\D$-valued vector fields $g$-othogonal to $X$, i.e.,	
	$$\<\ad_X^k(Z)\ |\ Z\in \Sec(\D),\quad g(Z,X)=0,\quad k=0,1,2,\hdots>$$
is of constant rank and does not contain $X$ at any point of $q(t)$, then the integral curves of $X$ are normal SR extremals. 
\end{itemize}
 
The first of the above results seems not to be present in the literature. Of course a characterization of abnormal extremals in Hamiltonian terms is well-known (see, e.g., \cite{Sussmann_coord_free_pmp}) and as such has been widely used. In many particular cases (see, e.g., \cite{Jakubczyk_Krynski_Pelletier_2009}) it allowed to obtain criteria to derive abnormal extremals similar to the smooth version of Theorem \ref{thm:abnormal}. The second of the above results appears in its smooth version in \cite{Liu_Sussmann_1995} for rank-2 distributions, and in the general version in \cite{Alcheikh_Orro_Pelletier_1997} in a formulation equivalent to ours. For these reasons we do not claim to be the first to obtain the above results (although our simple formulations using the language of flows of the optimal control seem to be new). What we believe, however, to be an added value is the simplicity of derivation of these results in our approach. Indeed, our proofs use only basic geometric tools, actually nothing more sophisticated than the definition of the flow, the derivation of the tangent space to a paraboloid, and the Gram-Schmidt algorithm. 

It should be stressed that the language of flows used throughout is much more effective, and in fact simpler, than the language of Lie brackets usually applied in the study of SR extremals. Indeed, the assertions of Theorems \ref{thm:abnormal} and \ref{thm:normal} are valid for non-smooth, i.e.,  absolutely continuous  curves and bounded measurable controls, do not require any regularity assumptions (contrary to the characterization in terms of Lie brackets) and work for single trajectories (not necessary families of trajectories).  

As an illustration of the above results we give a few examples. In particular, in Examples \ref{ex:geod} and \ref{ex:geod1} we were able to provide a surprisingly easy derivation of the Riemannian geodesic equation (obtaining the equation $\nabla_{\dot \gamma}\dot \gamma=0$ from the standard Hamiltonian approach is explained in \cite{Agrachev_Barilari_Boscain_2012,Sussmann_coord_free_pmp}). In Examples \ref{ex:234}, \ref{ex:zhit_nice}, and \ref{ex:heisenberg} we re-discover some results of \cite{Liu_Sussmann_1995} and \cite{Zhitomirskii_1995} concerning rank-2 distributions.  

\paragraph{Organization of the paper.}
We begin our considerations by a technical introduction in Section \ref{sec:technical}. Our main goal in this part is to introduce, in a rigorous way, natural differential geometric tools (Lie brackets, flows of time-dependent vector fields, distributions, etc.) in the non-smooth and time-dependent setting suitable for control theory (in general, we consider controls which are only locally bounded and measurable). Most of the results presented in this section are natural generalizations of the results well-known in the smooth case. They are essentially based on the local existence and uniqueness of solutions of ODE in the sense of Caratheodory (Theorem \ref{thm:solutions_ODE}). To avoid being too technical, we moved various parts of the exposition of this section (including some proofs and definitions) to Appendix \ref{sec:appendix}. 

In Section \ref{sec:contact}, we briefly recall basic definitions and constructions of contact geometry. In particular, we show an elegant construction of contact vector fields (infinitesimal symmetries of contact distributions) in terms of equivalence classes of vector fields modulo the contact distribution. This construction is more fundamental than the standard one in terms of generating functions (which requires a particular choice of a contact form). It seems to us that so far it has not been present in the literature. 

In Section \ref{sec:ptm}, we discuss in detail a natural contact structure on the projectivization of the cotangent bundle $\PP(\T^\ast M)$. In particular, we construct a natural contact transformation $\proj{F}$ of $\PP(\T^\ast M)$ induced by a diffeomorphism $F$ of $M$. Later we study an infinitesimal counterpart of this construction, i.e., a natural lift of a vector field $X$ on $M$ to a contact vector field $\cont{X}$ on $\PP(\T^\ast M)$. 

In Section \ref{sec:pmp}, we introduce the optimal control problem for a control system on a manifold $Q$ and formulate the PMP in its standard version (Theorem \ref{thm:pmp_hamiltonian}). Later we sketch the standard proof of the PMP introducing the cones $\wt\K_t$ and the separating hyperplanes $\wt\HH_t$. A proper interpretation of these objects, together with our previous considerations about the geometry of $\PP(\T^\ast M)$ from Section \ref{sec:ptm}, allows us to conclude Theorems \ref{thm:pmp_contact} and \ref{thm:pmp_covariant} which are the contact and the covariant versions of the PMP, respectively. 

Finally, in the last Section \ref{sec:appl}, we concentrate our attention on the geometry of the cones $\wt\K_t$ and hyperplanes $\wt\HH_t$ for the Riemannian and sub-Riemannian geodesic problems. The main results of that section, which characterize normal and abnormal SR extremals, were already discussed in detail in the paragraph ``Applications'' above. 
\section{Technical preliminaries}
\label{sec:technical}

As indicated in the Introduction, in this paper we shall apply the language of differential geometry to optimal control theory. This requires some attention as differential geometry uses tools such as vector fields, their flows, distributions and Lie brackets which are \emph{a priori} smooth, while in control theory it is natural to work with objects of lower regularity. The main technical difficulty is a rigorous introduction of the notion of the flow of a time-dependent vector field (TDVF, in short) with the time-dependence being, in general, only measurable. A solution of this problem, provided within the framework of chronological calculus, can be found in \cite{Agrachev_Sachkov_2004}. The recent monograph \cite{Jafarpour_Lewis_2014} with a detailed discussion of regularity aspects is another exhaustive source of information about this topic. 

Despite the existence of the above-mentioned excellent references, we decided to present our own explication of the notion of the flow of a TDVF. The reasons for that decision are three-fold. First of all this makes our paper self-contained. Secondly, we actually do not need the full machinery of \cite{Agrachev_Sachkov_2004} or \cite{Jafarpour_Lewis_2014}, so we can present a simplified approach. Finally, for future purposes we need to concentrate our attention on some specific aspects (such as the transport of a distribution along an integral curve of a TDVF and the relation of this transport with the Lie bracket) which are present in neither \cite{Agrachev_Sachkov_2004}, nor \cite{Jafarpour_Lewis_2014}. Our goal in this section is to give a minimal yet sufficient introduction to the above-mentioned concepts. We move technical details and rigorous proofs to Appendix \ref{sec:appendix}. 

\paragraph{Time-dependent vector fields and their flows.}

 Let $M$ be a smooth manifold. By a \emph{time-dependent vector field} on $M$ (\emph{TDVF}, in short) we shall understand a family of vector fields $X_t\in\X(M)$ parametrized by a real parameter $t$ (the \emph{time}). Every such a field defines the following non-autonomous ODE\footnote{Sometimes it is convenient to identify a TDVF $X_t$ on $M$ with the vector field $\wtilde X(x,t)=X_t(x)+\pa_t$ on $M\times\R$. Within this identification equation \eqref{eqn:ODE_1} is an $M$-projection of the autonomous ODE $(\dot x,\dot t)=\wtilde X(x,t)$ defined on $M\times\R$.} on $M$
\begin{equation}
\label{eqn:ODE_1}
\dot x(t)=X_t(x(t))\ .
\end{equation} 
A technical assumption that the map $(x,t)\mapsto X_t(x)$ is \emph{\nice} in the sense of Definition \ref{def:nice_map} below guarantees that solutions of \eqref{eqn:ODE_1} (in the sense of Caratheodory) locally exist, are unique and are \emph{absolutely continuous with bounded derivatives} (\emph{ACB}, in short, see Appendix \ref{sec:appendix}) with respect to the time $t$. For this reason from now on we shall restrict our attention to TDVF's $X_t$ satisfying the above assumption. We will call them \emph{\nice\ TDVF}'s. 
In a very similar context the notion of a Caratheodory section was introduced in the recent monograph \cite{Jafarpour_Lewis_2014}. Actually, in the language of the latter work our notion of a Caratheodory TDVF would be called a locally bounded Caratheodory vector field of class $C^1$. 

A solution of \eqref{eqn:ODE_1} with the initial condition $x(t_0)=x_0$ will be denoted by $x(t;t_0,x_0)$ and called an \emph{integral curve} of $X_t$. When speaking about families of such solutions with different initial conditions it will be convenient to introduce (local) maps $\A{tt_0}:M\ra M$ defined by $\A{tt_0}(x_0):=x(t;t_0,x_0)$. 
\begin{lemma}\label{lem:TDVF_technical}
Let $X_t\in\X(M)$ be a \nice\ TDVF on $M$. Then
\begin{itemize}
	\item For $t$ close enough to $t_0$ the maps $\A{tt_0}:M\ra M$ are well-defined local diffeomorphisms. 
	\item Moreover, they satisfy the following properties
	\begin{equation}
\label{eqn:t_flow}
\A{t_0t_0}=\id_M\quad \text{and}\quad \A{t\tau}(\A{\tau t_0})=\A{t t_0}\ ,
\end{equation}
whenever both sides are defined. 
\end{itemize}
\end{lemma}
Since $X_t$ is \nice, it satisfies locally the assumptions of Theorem \ref{thm:solutions_ODE}. Now the justification of Lemma \ref{lem:TDVF_technical} follows directly from the latter result. Properties \eqref{eqn:t_flow} are merely a consequence of the fact that $t\mapsto\A{tt_0}(x_0)$ is an integral curve of $X_t$. 

\begin{definition}
\label{def:td_flow} The family of local diffeomorphisms $\A{t\tau}:M\ra M$ described in the above lemma will be called the \emph{time-dependent flow} of $X_t$ (TD flow, in short).
\end{definition} 

Clearly $\A{tt_0}$ is a natural time-dependent analog of the notion of the flow of a vector field. This justifies the name ``TD flow''. It is worth noticing that, alike for the standard notion of the flow, there is a natural correspondence between TD flows and \nice\ TDVF's. 

\begin{lemma}
\label{lem:exist_TDVF}
Let $\A{t\tau}:M\ra M$ be a family of local diffeomorphisms satisfying \eqref{eqn:t_flow} and such that for each choice of $x_0\in M$ and $t_0\in\R$ the map $t\mapsto \A{tt_0}(x_0)$ is ACB. Then $\A{t\tau}$ is a TD flow of some \nice\ TDVF $X_t$. 
\end{lemma}

The natural candidate for such a TDVF is simply $X_t(x):=\frac{\pa}{\pa \tau}\big|_{\tau=t}\A{\tau t}(x)$. The remaining details are left to the reader. 

\paragraph{Distributions along integral curves of TDVF's.}
In this paragraph we shall introduce basic definitions and basic properties related with distributions defined along a single ACB integral curve $x(t)=x(t;t_0,x_0)$ (with $t\in[t_0,t_1]$) of a \nice\ TDVF $X_t$. In particular, for future purposes it will be crucial to understand the behavior of such distributions under the TD flow $\A{t\tau}$ of $X_t$. 

\begin{definition} 
\label{def:distib_at_all}
Let $x(t)=x(t;t_0,x_0)$ with $t\in[t_0,t_1]$ be an integral curve of a \nice\ TDVF $X_t$. 
A \emph{distribution $\B$ along $x(t)$} is a family of linear subspaces  $\B_{x(t)}\subset \T_{x(t)} M$ attached at each point of the considered curve. In general, the dimension of $\B_{x(t)}$ may vary from point to point. 

By an \emph{ACB section of $\B$} we will understand a vector field $Z$ along $x(t)$ such that $Z(x(t))\in \B_{x(t)}$ for every $t\in[t_0,t_1]$ and that the map $t\mapsto Z(x(t))$ is ACB. The space of such sections will be denoted by $\Sec_{ACB}(\B)$. A distribution $\B$ along $x(t)$ shall be called \emph{\sexy}\ if it is point-wise spanned by a finite set of elements of $\Sec_{ACB}(\B)$.

We shall say that $\B$ is \emph{$\A{t\tau}$-invariant} (or \emph{respected by a TD flow $\A{t\tau}$}) \emph{along $x(t)$} if 
$$\B_{x(t)}=\T \A{t\tau}(\B_{x(\tau)})$$
for every $t, \tau\in[t_0,t_1]$. Equivalently, $\B_{x(t)}=\T \A{tt_0}(\B_{x(t_0)})$ for every $t\in[t_0,t_1]$. In particular, if $\B$ is respected by $\A{t\tau}$ along $x(t)$ then it is of constant rank along $x(t)$. This follows from the fact that each map $\A{t\tau}$ is a local diffeomorphism.
\end{definition}

Let us remark that the idea behind the notion of a \sexy\ distribution is to provide a natural substitution of the notion of smoothness in the situation where a distribution is considered along a non-smooth curve. Observe namely that a restriction of a smooth vector field on $M$ to an ACB curve $x(t;t_0,x_0)$ is \emph{a priori} only an ACB vector field along $x(t;t_0,x_0)$.  
\begin{proposition}
\label{prop:sexy_examples}
\Sexy\ distributions appear naturally in the following two situations:
\begin{itemize}
	\item A restriction of a locally finitely generated smooth distribution on $M$ to an ACB curve $x(t)=x(t;t_0,x_0)$ is \sexy.
	\item Let $\A{t\tau}$ be the TD flow of a \nice\ TDVF $X_t$ and let $\B$ be a distribution along an integral curve $x(t)=x(t;t_0,x_0)$ of $X_t$. Then if $\B$ is $\A{t\tau}$-invariant along $x(t)$, it is also \sexy.
\end{itemize} 
\end{proposition}
The justification of the above result is straightforward. Regarding the first situation it was already observed that a restriction of a smooth vector field to an ACB curve is an ACB vector field. 
In the second situation, the distribution $\B$ is spanned by vector fields $\T\A{tt_0}(X^i)$ with $i=1,\hdots,k$, where $\{X^1,\hdots,X^k\}$ is any basis of $\B_{x_0}$. By the results of Lemma \ref{lem:bracket_single} these fields are ACB.

Given a distribution $\B$ along $x(t)$ we can always extend it to the smallest (with respect to inclusion) distribution along $x(t)$ containing $\B$ and respected by the TD flow $\A{t\tau}$ along $x(t)$. This construction will play a crucial role in geometric characterization of normal and abnormal SR extremals in Section \ref{sec:appl}. 

\begin{proposition}
\label{prop:AB}
 Let $x(t)=x(t;t_0,x_0)$ with $t\in[t_0,t_1]$ be a trajectory of a TDVF $X_t$. Let $\A{t\tau}$ be the TD flow of $X_t$ and let $\B$ be a distribution along $x(t)$. Then 
$$\A{\bullet}(\B)_{x(t)}:=\vect\{\T \A{t\tau}(X)\ |\ X\in\B_{x(\tau)},\quad t_0\leq\tau\leq t_1\}$$
is the smallest distribution along $x(t)$ which contains $\B$ and is respected by the TD flow $\A{t\tau}$ along $x(t)$.
\end{proposition}
Obviously, any distribution $\A{t\tau}$-invariant along $x(t)$ and containing $\B_{x(t)}$ must contain $\A{\bullet}(\B)_{x(t)}$. The fact that the latter is indeed $\A{t\tau}$-invariant along $x(t)$ follows easily from property \eqref{eqn:t_flow}.

\paragraph{Lie brackets and distributions.}
Constructing distributions $\A{t\tau}$-invariant along $x(t)$ introduced in Proposition \ref{prop:AB}, although conceptually very simple, is not very useful from the practical point of view, as it requires calculating the TD flow $\A{t\tau}$. This difficulty can be overcome by passing to an infinitesimal description in terms of the Lie brackets, however, for a price of loosing some generality. In this  paragraph we shall discuss this and some related problems in detail.    

\begin{definition}
\label{def:bracket}
Let $X_t$ be a \nice\ TDVF and $x(t)=x(t;t_0,x_0)$ its integral curve. Given any smooth vector field $Z\in \mathfrak{X}(M)$ we define \emph{the Lie bracket of $X_t$ and $Z$ along $x(t)$} by the formula 
$$[X_t,Z]_{x(t)}\ ,$$ 
i.e., we calculate the standard Lie bracket $[X_t,Z]$ ``freezing'' the time $t$ and then evaluate it at the point $x(t)$, thus obtaining a well-defined field of vectors along $x(t)$ (the regularity of the map $t\mapsto [X_t,Z]_{x(t)}$ is a separate issue that we shall discuss later).
\end{definition}
For future purposes we would like to extend Definition \ref{def:bracket} to be able to calculate the bracket $[X_t,Z]_{x(t)}$ also for fields $Z$ of lower regularity. That can be done, but at a price that the bracket $[X_t,Z]_{x(t)}$ would be defined only for almost every (a.e., in short) $t\in[t_0,t_1]$. The details of this construction are provided below.  

As a motivation recall that for $M=\R^n$, given two smooth vector fields $X,Z\in \X(\R^n)$ (understood as maps $X,Z:\R^n\ra\R^n$) their Lie bracket at a point $x_0$ equals
$$[X,Z]_{x_0}=\frac{\pa}{\pa t}\Big|_{t=0}Z(x(t))-\frac{\pa}{\pa s}\Big|_{s=0}X(z(s))\ ,$$ 
where $t\mapsto x(t)$ is the integral curve of $X$ emerging from $x_0$ at time $0$ (in particular, $\frac{\pa}{\pa t}\big|_{0}x(t)=X(x_0)$) and $s\mapsto z(s)$ is the integral curve of $Z$ emerging from $x_0$ at time $0$ (in particular, $\frac{\pa}{\pa s}\big|_0z(s)=Z(x_0)$). The above formula, actually, allows to define $[X,Z]_{x_0}$ on any smooth manifold $M$, simply by taking it as the definition of the Lie bracket $[X,Z]_{x_0}$ in a particular local coordinate system on $M$. It is an easy exercise to show that $[X,Z]_{x_0}$ defined in such a way is a true geometric object (i.e., it does not depend on the particular choice of a local chart).  
Note that in order to calculate $[X,Z]_{x_0}$ we need only to know $X$ along $s\mapsto z(s)$ and $Z$ along $t\mapsto x(t)$. 

Observe that to use directly the above computational definition to calculate the Lie bracket $[X_t,Z]_{x(t)}$ along $x(t)=x(t;t_0,x_0)$ we should use a separate integral curve of the field $X_t$ (with ``frozen'' time) for every $t\in[t_0,t_1]$, i.e., 
\begin{equation}
\label{eqn:lie_bracket}
[X_t,Z]_{x(t)}=\frac{\pa}{\pa \tau}\Big|_{t}Z(x_t(\tau))-\frac{\pa}{\pa s}\Big|_{s=0}X_t(z(s,t))\ ,
\end{equation} 
where for each $t\in[t_0,t_1]$ the map $\tau\mapsto x_t(\tau)$ is the integral curve of the field $X_t$ emerging from the point  $x(t)$ at time $\tau=t$ and $s\mapsto z(s,t)=z(s;0,x(t))$ is the integral curve of $Z$ emerging from $x(t)$ at $s=0$, i.e., $z(t,0)=x(t)=x(t;t_0,x_0)$ and $\frac{\pa}{\pa s}\big|_0z(t,s)=Z(x(t))$. Observe now that by definition $\dot x_t(t)=X_t(x(t))=\dot x(t)$ and thus \eqref{eqn:lie_bracket} holds for $x_t(\tau)=x(\tau)$. What is more, \eqref{eqn:lie_bracket} is well-defined at a given time $t\in[t_0,t_1]$ also for any vector field $Z$ on $M$ (not necessarily smooth) such that the map $\tau\mapsto Z(x(\tau))$ is differentiable at $\tau=t$. This observation justifies the following statement
\begin{proposition}\label{prop:bracket_1}
Assuming that $t\mapsto Z(x(t))$ is an ACB map and that $X_t$ is a \nice\ TDVF, the Lie bracket $[X_t,Z]_{x(t)}$ is defined by formula \eqref{eqn:lie_bracket} almost everywhere along $x(t)$. In fact, it is well-defined at all regular points of $t\mapsto Z(x(t))$. Moreover, $t\mapsto [X_t,Z]_{x(t)}$ is a measurable and locally bounded map.

The Lie bracket $[X_t,Z]_{x(t)}$, is completely determined by the values of $Z$ along $x(t)$ and by the values of $X_t$ in a neighborhood of $x(t)$. 
\end{proposition}

In other words, formula \eqref{eqn:lie_bracket} is an extension of Definition \ref{def:bracket}
which allows to calculate the Lie bracket $[X_t,Z]_{x(t)}$ at almost every point of a given integral curve $x(t)$ of $X_t$, for vector fields $Z$ defined only along $x(t)$ and such that $t\mapsto Z(x(t))$ is ACB. The latter generalization is necessary in control theory, since, as $t\mapsto x(t)$ is in general ACB only, even if $Z$ is a smooth vector field, we cannot expect the map $t\mapsto Z(x(t))$ to be of regularity higher than ACB. 

The above construction of the Lie bracket $[X_t,Z]_{x(t)}$ allows to introduce the following natural construction.

\begin{definition}
\label{def:bracket_distrib}
Let $X_t$ be a \nice\ TDVF, $x(t)=x(t;t_0,x_0)$ (with $t\in[t_0,t_1]$) its integral curve and let $\B$ be a distribution along $x(t)$. By $[X_t,\B]$ we shall understand the distribution along $x(t)$ generated by the Lie brackets of $X_t$ and all ACB sections of $\B$:
$$[X_t,\B]_{x(t)}:=\vect\{[X_t,Y]_{x(t)}\ |\ Y\in\Sec_{ACB}(\B)\}\ ,$$ 
where we consider $[X_t,Y]_{x(t)}$ at all points where it makes sense, i.e., at which the bracket $[X_t,Y]_{x(t)}$ is well-defined. 

A \sexy\ distribution $\B$ along $x(t)$ will be called \emph{$X_t$-invariant along $x(t)$}  if
$$[X_t,\B]_{x(t)}\subset\B_{x(t)}\quad\text{for almost every $t\in[t_0,t_1]$.}$$
\end{definition}

Note that neither $[X_t,\B]$ nor $\B+[X_t,\B]$ need be \sexy\ distributions along $x(t)$ even if so was $\B$ as, in general, there is no guarantee that these distributions will be spanned by ACB sections (we can loose regularity when calculating the Lie bracket).   

The following result explains the relation between the $\A{t\tau}$- and $X_t$-invariance of distributions along $x(t)$.  

\begin{theorem}
\label{thm:flow_bracket}
Let $\B$ be a distribution along $x(t)=x(t;t_0,x_0)$ (with $t\in[t_0,t_1]$), an integral curve of a \nice\ TDVF $X_t$, and let $\A{t\tau}$ be the TD flow of $X_t$. The following conditions are equivalent
\begin{enumerate}[(a)] 
	\item \label{cond:a} $\B$ is respected by the TD flow $\A{t\tau}$ of $X_t$ along $x(t)$.
	\item \label{cond:b} $\B$ is a \sexy\ distribution $X_t$-invariant and of constant rank along $x(t)$. 
\end{enumerate} 
\end{theorem}
The proof is given in Appendix \ref{sec:appendix}. Note that the equivalence between $X_t$- and $\A{t\tau}$-invariance is valid only if the considered distribution $\B$ along $x(t)$ satisfies regularity conditions: it has to be \sexy\ and of constant rank along $x(t)$. 

Given a \sexy\ distribution $\B$ along $x(t)$, it is clear in the light of the above result,  that $\A{\bullet}(\B)_{x(t)}$, the smallest distribution $\A{t\tau}$-invariant along $x(t)$ and containing $\B$, should be closed under the operation $[X_t,\cdot]$. Thus, in the smooth case, it is natural to try to construct $\A{\bullet}(\B)$ in the following way. 

\begin{lemma}
\label{cor:algorithm}
Let $X$ be a $C^\infty$-smooth vector field and let $\B$ a $C^\infty$-smooth distribution on $M$. Assume that along an integral curve $x(t)=x(t;t_0,x_0)$ of $X$ (with $t\in[t_0,t_1]$), the distribution spanned by the iterated Lie brackets of $X$ with all possible $\B$-valued vector fields, i.e.,
\begin{equation}
\label{eqn:Y_B}
\ad^\infty_X(\B)_{x(t)}:=\<\ad_X^k(Z)_{x(t)}\ |\ Z\in\Sec(\B),\quad k=0,1,2,\hdots>
\end{equation}
is of constant rank along $x(t)$. Then $\ad^\infty_X(\B)_{x(t)}$ is the smallest distribution along $x(t)$ containing $\B_{x(t)}$ and respected by $\A{t}$, the flow of $X$, i.e., $\ad^\infty_X(\B)_{x(t)}=\A{\bullet}(\B)_{x(t)}$.  
\end{lemma}

\begin{proof}
The justification of the above result is quite simple. By construction, $\ad^\infty_X(\B)_{x(t)}$ is the smallest distribution along $x(t)$ containing $\B_{x(t)}$ and closed under the operation $\ad_X=[X,\cdot]$. It is clear that $\ad^\infty_X(\B)$ is spanned by a finite number of smooth vector fields of the form $\ad_X^k(Z)$, where $Z\in\Sec(\B)$, and thus it is \sexy. Since it is also of constant rank along $x(t)$ we can use Theorem \ref{thm:flow_bracket} (for a time-independent vector field $X$) to prove that $\ad^\infty_X(\B)_{x(t)}$ is invariant along $x(t)$ under the flow $\A{t}$. We conclude that $\A{\bullet}(\B)_{x(t)}\subset \ad^\infty_X(\B)_{x(t)}$. On the other hand, since $\A{\bullet}(\B)_{x(t)}$ is $\A{t}$-invariant along $x(t)$, again by Theorem \ref{thm:flow_bracket}, it must be closed with respect to the operation $[X,\cdot]$. In particular, it must contain the smallest distribution along $x(t)$ containing $\B_{x(t)}$ and closed under the operation $[X,\cdot]$. Thus $\A{\bullet}(\B)_{x(t)}\supset \ad^\infty_X(\B)_{x(t)}$. This ends the proof. 
\end{proof} 

\begin{remark} 
Let us remark that the construction provided by \eqref{eqn:Y_B} would be, in general, not possible in all non-smooth cases. The basic reason is that the Lie bracket defined by \eqref{eqn:lie_bracket} is of regularity lower than the initial vector fields, i.e., $[X_t,Z]$ may not be ACB along $x(t)$ even if so were $X_t$ and $Z$. Thus by adding the iterated Lie brackets to the initial distribution $\B$, we may loose the property that it is \sexy\ (cf. also a remark following Definition \ref{def:bracket_distrib}) which is essential for Theorem \ref{thm:flow_bracket} to hold. 

Also the constant rank condition is important, as otherwise the correspondence between $X_t$- and $\A{t\tau}$-invariance provided by Theorem \ref{thm:flow_bracket} does not hold. If \eqref{eqn:Y_B} is not of constant rank along $x(t)$ we may only say that $\ad^\infty_X(\B)_{x(t)}\subset \A{\bullet}(\B)_{x(t)}$ (see also Remark \ref{rem:zhit_wrong}). 

It is worth noticing that this situation resembles the well-known results of Sussmann \cite{Sussmann_1973} concerning the integrability of distributions: being closed under the Lie bracket is not sufficient for integrability, as the invariance with respect to the flows of distribution-valued vector fields is also needed. After adding an extra assumption that the rank of the distribution is constant, the latter condition can be relaxed. 
\end{remark} 

By the results of Proposition \ref{prop:bracket_1}, the property that a distribution $\B$ is $X_t$-invariant along $x(t)$ depends not only on $\B$ and the values of a \nice\ TDVF $X_t$ along $x(t)$, but also on the values of $X_t$ in a neighborhood of that integral curve. It turns out, however, that in a class of natural situations the knowledge of $X_t$ along $x(t)$ suffices for checking the $X_t$-invariance.

\begin{lemma}
\label{lem:along_x}
Let $\D$ be a smooth distribution of constant rank on $M$, $X_t$ a \nice\ $\D$-valued TDVF and 
$x(t)=x(t;t_0,x_0)$ (with $t\in[t_0,t_1]$) an integral curve of $X_t$. Let $\B$ be a \sexy\ distribution along $x(t)$, such that $\D_{x(t)}\subset\B_{x(t)}$ for every $t$. Then the property of $\B$ being $X_t$-invariant along $x(t)$ depends only on the values of $X_t$ along $x(t)$.  
\end{lemma}

The proof is given in Appendix \ref{sec:appendix}.

\section{The basics of contact geometry}
\label{sec:contact}

\paragraph{Contact manifolds and contact transformations.}
In this section we shall recall basic facts from contact geometry. A contact structure on a manifold $\M$ is a smooth co-rank one distribution $\C\subset \T\M$ satisfying certain maximum non-degeneracy condition. In order to formalize that condition we introduce the following geometric construction. From now on we shall assume that the pair $(\M,\C)$ consists of a smooth manifold $\M$ and a smooth co-rank one distribution $\C$ on $\M$. Sometimes it will be convenient to treat $\C$ as a vector subbundle of $\T\M$.

Given $(\M,\C)$ one can define the\emph{ bundle normal to $\C$} in $\T\M$ as the quotient 
$$\N\C:=\T\M/\C\ .$$
Note that $\N\C$ has a natural structure of a line bundle (since $\C$ is of co-rank one) over $\M$. We shall denote this bundle by $\tau:\N\C\ra \M$. 

Let now $X$ and $Y$ be two $\C$-valued vector fields on $\M$. It is easy to check that the class of their Lie bracket $[X,Y]$ in $\N\C$ is tensorial with respect to both $X$ and $Y$. That is,
for any pair of smooth functions $\phi,\psi\in C^\infty(\M)$ 
$$[\phi\cdot X,\psi\cdot Y]\equiv \phi\psi\cdot[X,Y]\mod\C\ .$$ 
It follows that the assignment 
$$(X,Y)\longmapsto\beta(X,Y):=[X,Y]\mod \C \ ,$$
defines a $\N\C$-valued 2-form $\beta:\Lambda^2\C\ra \N\C\ .$ Now we are ready to state the following

\begin{definition}
\label{def:cont_mfd}
A pair $(\M,\C)$ consisting of a smooth manifold $\M$ and a smooth co-rank one distribution $\C\subset\T\M$ is called a \emph{contact manifold} if the associated $\N\C$-valued 2-form $\beta$ is non-degenerate, i.e., if $\beta(X,\cdot)\equiv 0$ implies $X\equiv0$. 

Sometimes we call $\C$ a \emph{contact structure} or a \emph{contact distribution} on $\M$.
\end{definition}

Observe that $\C$ is necessarily of even-rank ($\M$ is odd-dimensional). This follows from a simple fact from linear algebra that every skew-symmetric 2-form on an odd-dimensional space has a non-trivial kernel. 

\begin{definition}
\label{def:cvf}
Let $(\M,\C)$ be a contact manifold. A diffeomorphism $F:\M\ra\M$ which preserves the contact distribution, i.e., $\T F(\C_p)=\C_{F(p)}$ for every $p\in \M$, where $\T F$ stands for the tangent map of $F$, is called a \emph{contact transformation}.
By a \emph{contact vector field} (\emph{CVF} in short) on $\M$ (or an \emph{infinitesimal symmetry} of $(\M,\C)$) we shall understand a smooth vector field $X\in\X(\M)$ preserving the contact distribution $\C$, i.e., 
$$[X,\C]\subset\C\ .$$
Equivalently, $X$ is a CVF if and only if its (local) flow $\A{t}$ consists of contact transformations (cf. Theorem~\ref{thm:flow_bracket}).  
\end{definition}

It is worth mentioning that the above relation between contact vector fields and flows consisting of contact transformations can be generalized to the context of TDVF's and TD flows (cf. Section \ref{sec:technical}). We will need this generalized relation in Section \ref{sec:pmp} after introducing control systems. 

\begin{proposition}
\label{prop:contact_evolution}
Let $X_t$ be a \nice\ TDVF on a contact manifold $(\M,\C)$ and let $\A{t\tau}$ be the TD flow of $X_t$. Then $X_t$ is a contact vector field for every $t\in\R$ (i.e., $[X_t,\C]\subset\C$) if and only if the  TD flow $\A{t\tau}$ consists of contact transformations.  
\end{proposition}
The proof follows directly from Theorem \ref{thm:flow_bracket} by taking $\B=\C$ (which is \sexy\ -- cf. Proposition \ref{prop:sexy_examples}).

\paragraph{Characterization of CVF's.}
It turns out that there is a one-to-one correspondence between CVF's on $\M$ and sections of the normal bundle $\N\C$. 

\begin{lemma}
\label{lem:cvf}
Let $X\in\X(\M)$ be any representative of the class $[X]\in\Sec(\N\C)$. By $h(X)\in\Sec(\C)$ we shall denote the unique $\C$-valued vector field satisfying 
$$[h(X), Y]=[Y,X] \mod \C$$
for every $Y\in \Sec(\C)$. 
The assignment $[X]\mapsto \sym{[X]}:=X+h(X)$ is well-defined and establishes a one-to-one correspondence between 
sections of $\N\C$ and CVF's on $\M$. 
\end{lemma}

\begin{remark}
Throughout we will denote by $X\in\X(\M)$ vector fields on $\M$, by $Y\in\Sec(\C)$ vector fields valued in $\C$ and by $C$ (also with variants, like $\sym{\phi}$, $\sym{[X]}$ or $\cont{\phi}$) contact vector fields.  
\end{remark}

\begin{proof}
Let us begin with introducing the following geometric construction. With every smooth vector field $X\in\X(\M)$ one can associate a $\N\C$-valued 1-form $\alpha_X:\C\ra\N\C$ defined by the formula
$$\alpha_X(Y):=[Y,X]\mod\C\ ,$$ 
where $Y$ is a $\C$-valued vector field. The correctness of this definition follows from the fact that for every $\C$-valued vector field $Y$ and for any function $\phi\in C^\infty(\M)$ we have 
$$[\phi\cdot Y,X]\equiv \phi\cdot[Y,X]\mod\C\ . $$

Using the one-form $\alpha_X$ we can prove that $[X]\mapsto\sym{[X]}$ is a well-defined map, i.e., that the value of $\sym{[X]}$ does not depend on the choice of the representative $X$. Indeed, we can interpret $h(X)$ as the unique (note that $\beta$ is non-degenerate) solution of the equation $\alpha_X(\cdot)=\beta(h(X),\cdot)$. Now observe that if $X$ and $X'$ are two different representatives of $[X]$, then $Y:=X'-X$ is a $\C$-valued vector field on $\M$. Thus we have $\alpha_Y(\cdot)=-\beta(Y,\cdot)$ and hence, using the obvious linearity of $\alpha_X$ with respect to $X$, we get
$$\beta(h(X'),\cdot)=\alpha_{X'}(\cdot)=\alpha_{X+Y}(\cdot)=\alpha_X(\cdot)+\alpha_Y(\cdot)=\beta(h(X),\cdot)-\beta(Y,\cdot)=\beta(h(X)-Y,\cdot) \ .$$
We conclude that $h(X')=h(X)-Y$ and, consequently, 
$$\sym{[X']}=X'+h(X')=X+Y+h(X)-Y=X+h(X)=\sym{[X]}\ .$$
\smallskip

Secondly, observe that $\sym{[X]}$ is a CVF. Indeed, by construction, given any $\C$-valued vector field $Y$ we have 
$$[\sym{[X]},Y]=[X+h(X),Y]\equiv 0\mod\C\ ,$$ 
i.e., $[\sym{[X]},Y]$ is a $\C$-valued vector field.
\medskip

Finally, we need to check that every CVF is of the form $\sym{[X]}$. By construction the class of $\sym{[X]}$ in $\N\C$ is equal to the class of $X$ in $\N \C$ (these two vector fields differ by a $\C$-valued vector field $h(X)$). Thus the classes of CVF's of the form $\sym{[X]}$ realize every possible section of $\N\C$. Now it is enough to observe that the $\N\C$-class uniquely determines a CVF. Indeed, if $C$ and $C'$ are two CVF's belonging to the same class in $\N\C$, then their difference $X-X'$ is a $\C$-valued CVF, i.e., $[C-C',Y]\equiv 0\mod \C$ for any $\C$-valued vector field $Y$. That is, $\beta(C-C',\cdot)\equiv 0$ and from the non-degeneracy of $\beta$ we conclude that $C-C'\equiv 0$. 
This ends the proof.  \end{proof}

\begin{remark}
\label{rem:generator}
It is natural to call a vector field $X\in \X(\M)$ (or its $\N\C$-class $[X]$) a \emph{generator} of the CVF $\sym{[X]}$. Observe that the $\N\C$-class of the CVF $\sym{[X]}$ is the same as the class of its generator $X$ (they differ by a $\C$-valued vector field $h(X)$). 
\end{remark}

In the literature, see e.g., \cite{Marle_Liber_sympl_geom_anal_mech_1987}, a contact distribution $\C$ on $\M$ is often presented as the kernel of a certain 1-form $\omega\in\Lambda^1(\M)$ (such an $\omega$ is then called a \emph{contact form}). In the language of $\omega$, the maximum non-degeneracy condition can be expressed as the non-degeneracy of the 2-form $\dd\omega$ on $\C$. The latter is equivalent to the condition that $\omega\wedge(\dd\omega)^{\wedge n}$, where $n=\frac 12\rank \C$, is a volume form on $\M$ (i.e., $\omega\wedge(\dd\omega)^{\wedge n}\neq 0$). 

Also CVF's have an elegant characterization in terms of contact forms. One can show that CVF's are in one-to-one correspondence with smooth functions on $\M$. Choose a contact form $\omega$ such that $\C=\ker\omega$, then this correspondence is given by an assignment $\phi\mapsto \sym{\phi}$, where $\sym{\phi}$ is the unique vector field on $M$ such that $\omega(\sym{\phi})=\phi$ and $(\sym{\phi}\lrcorner\dd\omega)|_{\C}=-\dd \phi|_{\C}$. A function $\phi$ is usually called the \emph{generating function} of the corresponding CVF $\sym{\phi}$ associated with the contact form $\omega$. Notice that given a contact vector field $\sym{}=\sym{\phi}$ and a contact form $\omega$ one can recover the generating function simply by evaluating $\omega$ on $\sym{}$, i.e., $\phi=\omega(\sym{})$. 

It is interesting to relate the construction $\phi\mapsto \sym{\phi}$ with the construction $[X]\mapsto \sym{[X]}$ given above. Namely, the choice of a contact form $\omega$ allows to introduce a vector field $R\in\X(\M)$ (known as the \emph{Reeb vector field}) defined uniquely by the conditions $\omega(R)=1$ and $R\lrcorner\dd\omega=0$. Since $R$ is not contained in $\C=\ker\omega$, its class $[R]$ establishes a basis of the normal bundle $\N\C$. Consequently, we can identify smooth functions on $\M$ with sections of $\N\C$, via $\phi\mapsto [\phi R]$. Now it is not difficult to prove that $\sym{\phi}=\sym{[\phi R]}$ and conversely that $\sym{[X]}=\sym{\phi}$ for $\phi=\omega(\sym{[X]})=\omega(X)$. The details are left to the reader. 

Note, however, that the description of the contact distribution $\C$ in terms of a contact form $\omega$ is, in general, non-canonical (as every rescaling of $\omega$ by a nowhere-vanishing function gives the same kernel $\C$) and valid only locally (as there clearly exist contact distributions which cannot be globally presented as kernels of single 1-forms). For this reason the description of a contact manifold $(\M,\C)$ in terms of $\C$ and related objects (e.g., $\N\C$, $\beta$) is more fundamental and often conceptually simpler (for example, in the description of CVF's) than the one in terms of $\omega$. Not to mention that, for instance, the construction of the CVF $\sym{\phi}$ does depend on the particular choice of $\omega$, whereas the construction of $\sym{[X]}$ is universal. 

\begin{remark}
\label{rem:cvf_as_aff_control}
In the context of CVF's on a contact manifold $(\M,\C)$ it is worth noticing an elegant correspondence between CVF's on $\M$ and a certain class of control-affine systems on $\M$. A \emph{control-affine  system} on a manifold $\M$ is usually understood as a differential equation of the form
\begin{equation}
\label{eqn:aff_cs}
\dot x=f(x)+\sum_{i=1}^mu^ig_i(x)\ .
\end{equation} 
Here $f,g_i\in\X(\M)$ are smooth vector fields ($f$ is usually called a \emph{drift}) and $(u^1,\hdots,u^m)^T\in\R^m$ are control parameters. Trajectories of the control system \eqref{eqn:aff_cs} are integral curves $\dot x(t)\in \mathcal{A}(x(t))$ of the affine distribution
$$\mathcal{A}=f+\mathcal{G}\subset\T\M\ ,$$
where $\mathcal{G}=\<g_1,\hdots,g_m>$. Note that the distribution $\mathcal{G}$, the linear part of $\mathcal{A}$, is well-defined, whereas the drift is defined only relative to $\mathcal{G}$, i.e., $f+\mathcal{G}=f'+\mathcal{G}$, whenever $f-f'\in\mathcal{G}$. In the light of our considerations about CVF's it is easy to prove the following fact.
\begin{proposition}
\label{prop:cfv_as_aff_cs}
Let $(\M,\C)$ be a contact manifold. There is a one-to-one correspondence between CVF's on $\M$ and control-affine systems (equivalently, affine distributions) on $\M$ of the form $\mathcal{A}=X+\C\subset\T\M$, where $X\in\Sec(\M)$. 
\end{proposition}
Indeed, to every CVF $\sym{}$, we attach the affine distribution (control-affine system) $\mathcal{A}=\sym{}+\C$. Conversely, given an affine distribution (control-affine system) $\mathcal{A}=X+\mathcal{C}$, there exists a unique CVF  $\sym{}\in\Sec(
\mathcal{A})$, namely $\sym{}=\sym{[X]}$, such that $\mathcal{A}=\sym{}+\C=\sym{[X]}+\C$. 
 In other words, on every contact manifold $(\M,\C)$, there are as many CVF's $\sym{}$'s as control-affine systems $\mathcal{A}=X+\C$'s, the correspondence being established by the map $\mathcal{A}=X+\C\mapsto \sym{[X]}$. 
\end{remark}
\section{Contact geometry of $\PP(\T^\ast M)$}
\label{sec:ptm}

In this section we shall describe the natural contact structure on $\PP(\T^\ast M)$ and its relation with the canonical symplectic structure on $\T^\ast M$ (see, e.g., \cite{Arnold_1989} or \cite{Marle_Liber_sympl_geom_anal_mech_1987}). Later it will turn out that this structure for $M=Q\times\R$ plays the crucial role in optimal control theory.

\paragraph{The canonical contact structure on $\PP(\T^\ast M)$.}

Let us denote the cotangent bundle of a manifold $M$ by $\pi_M:\T^\ast M\ra M$. The projectivized cotangent bundle $\PP(\T^\ast M)$ is defined as the space of equivalence classes of non-zero covectors from $\T^\ast M$ with $[\theta]=[\theta']$ if $\pi_M(\theta)=\pi_M(\theta')$ and $\theta=a\cdot\theta'$ for some scalar $a\in\R\setminus\{0\}$. Clearly, $\PP(\T^\ast M)$ is naturally a smooth manifold and also a fiber bundle over $M$ with the projection $\pi:\PP(\T^\ast M)\ra M$ given by $\pi:[\theta]\mapsto \pi_M(\theta)$. The fiber of $\pi$ over $p\in M$ is simply the projective space $\PP(\T^\ast_p M)$. It is worth noticing that $\PP(\T^\ast M)$ can be also understood as the space of hyperplanes in $\T M$ (a \emph{manifold of contact elements}), where we can identify each point $[\theta]\in\PP(\T^\ast M)$ 
with the hyperplane $\HH_{[\theta]}:=\ker \theta\subset \T_{\pi_M(\theta)}M$. 

\begin{lemma}
\label{lem:PTM_contact}
$\PP(\T^\ast M)$ carries a canonical contact structure given by
\begin{equation}
\label{eqn:cont_tpm}
\C_{[\theta]}=\{Y\in \T_{[\theta]}\PP(\T^\ast M)\ |\ \T\pi(Y)\in\HH_{[\theta]}=\ker\theta\} \ ,
\end{equation}
for $[\theta]\in \PP(\T^\ast M)$. 
\end{lemma}

In other words, $\C_{[\theta]}$ consists of all vectors in $\T_{[\theta]}\PP(\T^\ast M)$ which project, under $\T\pi$, to the hyperplane $\HH_{[\theta]}=\ker\theta$, see Fig. \ref{fig:contact_PTM}. We shall refer to this contact structure as the \emph{canonical contact structure on $\PP(\T^\ast M)$}. 
\begin{figure}[ht]%
\begin{center}
\def\svgwidth{0.5 \columnwidth}
\input{contact_fig.tex}
\caption{The canonical contact structure on $\PP(\T^\ast M)$.}
\label{fig:contact_PTM}%
\end{center}
\end{figure}

The fact that \eqref{eqn:cont_tpm} defines a contact structure is well-known in the literature. The proof is given, for instance, in Appendix 4 of the book of Arnold \cite{Arnold_1989}, where the reasoning is based on the properties of the Liouville 1-form $\Omega_M$ on the cotangent manifold $\T^\ast M$. For  convenience of our future considerations in Section \ref{sec:pmp} we shall, however, present a separate proof quite similar to the one of Arnold. 

\begin{proof}
Let $R\in\X(M)$ be any smooth vector field. We shall now construct a contact form $\omega_R$ on an open subset of $\PP(T^\ast M)$ for which $R$ will be the Reeb vector field (cf. our considerations following Remark \ref{rem:generator}). The set $\U_R=\{[\theta]\ |\ \theta(R)\neq 0\}\subset \PP(\T^\ast M)$ is an open subset and $\U_R$ projects under $\pi$ to the open subset $\{p\ |\ R(p)\neq 0\}\subset M$. In the language of hyperplanes, $\U_R$ consists of all hyperplanes in $\T M$ which are transversal to the given field $R$. Clearly the collection of subsets $\U_R$ for all possible vector fields $R$ forms an open covering of $\PP(\T^\ast M)$. The open subset $\U_R\subset \PP(\T^\ast M)$ can be naturally embedded as a co-dimension one submanifold in $\T^\ast M$ by means of the map
$$i_R:\U_R\hookrightarrow \T^\ast M\ ,$$
which assigns to the class $[\theta]\in \U_R$ the unique representative $\theta$ such that $\theta(R)=1$.
 Clearly, the natural projection $\pi:\PP(\T^\ast M)\ra M$, restricted to $\U_R$, is simply the composition of $i_R:\U_R\ra\T^\ast M$ and the cotangent projection $\pi_M:\T^\ast M\ra M$, i.e.,
\begin{equation}
\label{eqn:pi_composition}
\pi\big|_{\U_R}=\pi_M\circ i_R\ . 
\end{equation}

Denote by $\Omega_M$ the Liouville form on $\T^\ast M$, i.e., $\Omega_M\big|_\theta(Y)=\<\theta,\T\pi_M(Y)>$, for $Y\in\T_\theta\T^\ast M$. We claim that the pullback $\omega_R:=(i_R)^\ast\Omega_M$ is a contact form on $\U_R\subset \PP(\T^\ast M)$. Indeed, by definition, $Y\in\T_{[\theta]}\PP(\T^\ast M)$ belongs to $\C_{[\theta]}$ if and only if $\T\pi(Y)\in\HH_{[\theta]}=\ker\theta$. In other words, $Y\in\C_{[\theta]}$ if and only if $\<\theta,\T\pi(Y)>=0$. By \eqref{eqn:pi_composition} for $\theta=i_R([\theta])$ we have  
\begin{equation}
\label{eqn:liouville}
\<\theta,\T\pi(Y)>=\<\theta,\T\pi_M(\T i_R(Y))>=\Omega_M\big|_{\theta}(\T i_R(Y))=(i_R)^\ast \Omega_M\big|_{[\theta]}(Y)=\omega_R\big|_{[\theta]}(Y)\ .
\end{equation}
We conclude that $\C|_{\U_R}=\ker\omega_R$.

To finish the proof it is enough to check that $\omega_R$ satisfies the maximum non-degeneracy condition. This can be easily seen by introducing local coordinates $(x^0,x^1,\hdots,x^n)$ on $M$ in which $R=\pa_{x^0}$ (recall that $R$ is non-vanishing on $\pi(\U_R)$, so such a choice is locally possible). Let $(x^i,p_i)$ be the induced coordinates on $\T^\ast M$. It is clear that in these coordinates the image $i_R(\U_R)\subset\T^\ast M$ is characterized by equation $p_0=1$ and thus the Liouville form $\Omega_M=\sum_{i=0}^np_i\dd x^i$ restricted to $i_R(\U_R)$ is simply $\dd x^0+\sum_{i=1}^np_i\dd x^i$. Obviously, the pullback functions $\widetilde x^i:=(i_R)^\ast x^i$ with $i=0,\hdots,n$ and
$\widetilde p_j:=(i_R)^\ast p_j$ with $j=1,\hdots,n$ form a coordinate system in $\U_R$. In these coordinates the form $\omega_R=(i_R)^\ast\Omega_M$ simply reads as $\dd \widetilde x^0+\sum_{i=1}^n\widetilde p_i\dd \widetilde x^i$. It is a matter of a simple calculation to check that such a one-form satisfies the maximum non-degeneracy condition. We conclude that $\omega_R$ is, indeed, a contact form on $\U_R$ for the canonical contact structure on $\PP(\T^\ast M)$.   
\end{proof}

\paragraph{Contact transformations of $\PP(\T^\ast M)$ induced by diffeomorphisms.}
In this paragraph we will define contact transformations of $\PP(\T^\ast M)$ which are natural lifts of diffeomorphisms of the base $M$. 

\begin{definition}
\label{def:PF}
Let $F:M\ra M$ be a diffeomorphism. Its tangent map $\T F:\T M\ra\T M$ induces a natural transformation $\proj{F}:\PP(\T^\ast M)\ra\PP(\T^\ast M)$ of the space of hyperplanes in $\T M$, i.e., given a hyperplane $\HH\subset\T_p M$, we define the hyperplane $\proj{F}(\HH)\subset\T_{F(p)}M$ to be simply the image $\T F(\HH)$. The map $\proj{F}$ shall be called the \emph{contact lift of $F$}.
\end{definition}
 Observe that if $\HH=\ker\theta$, then $\T F(\HH)=\ker((F^{-1})^\ast(\theta))$. In other words, $\proj{F}$ is the projection of $(F^{-1})^\ast:\T^\ast M\ra\T^\ast M$ to $\PP(\T^\ast M)$ (note that $(F^{-1})^\ast$ is linear on fibers of $\T^\ast M$, so this projection is well-defined)
$$\proj{F}([\theta])=[(F^{-1})^\ast\theta]\ .$$  
It is worth noticing that $\proj{F}$ respects the fiber bundle structure of $\pi:\PP(\T^\ast M)\ra M$:
\begin{equation}
\label{eqn:diag_PF}
\xymatrix{
\PP(\T^\ast M)\ar[rr]^{\proj{F}}\ar[d]^{\pi}&& \PP(\T^\ast M)\ar[d]^{\pi}\\
M\ar[rr]^{F}&& M \ . 
}\end{equation}

We claim that 
\begin{lemma}
\label{lem:PF_contact}
$\proj{F}$ is a contact transformation with respect to the canonical contact structure on $\PP(\T^\ast M)$. 
\end{lemma}

\begin{proof}
Let $Y$ be an element of $\T_{[\theta]}\PP(\T^\ast M)$ projecting to $\T\pi(Y)=:\ul Y$ under $\T\pi$. By diagram \eqref{eqn:diag_PF}, the tangent map $\T\proj{F}$ sends $Y$ to an element of $\T_{[(F^{-1})^\ast\theta]}\PP(\T^\ast M)$ lying over $\T F(\ul Y)$. 

Now if $Y$ belongs to the contact distribution $\C_{[\theta]}$, i.e., see \eqref{eqn:cont_tpm}, if $\ul Y\in \ker\theta$, then 
$$\<(F^{-1})^\ast\theta,\T F(\ul Y)>=\<\theta,\T F^{-1}\T F(\ul Y)>=\<\theta,\ul Y>=0\ .$$ 
Consequently, $\T F(\ul Y)\in \ker (F^{-1})^\ast\theta$, and thus $\T\proj{F}(Y)$ belongs to $\C_{[(F^{-1})^\ast\theta]}$, which ends the proof.  
\end{proof}

Let us remark that an alternative way to prove the above result is to show that $(F^{-1})^\ast$ maps the contact form $\omega_R$ to $\omega_{\T F(R)}$. To prove that, one uses the fact that the pullback $(F^{-1})^\ast$ preserves the Liouville form. 

\paragraph{CVF's on $\PP(\T^\ast M)$ induced by base vector fields.}
The results of the previous paragraph have their natural infinitesimal version. 
\begin{definition}
\label{def:PX}
Let $X\in\X(M)$ be a smooth vector field. By the \emph{contact lift of $X$} we shall understand the contact vector field $\cont{X}$ on $\PP(\T^\ast M)$ whose flow is $\proj{\A{t}}$, the contact lift of the flow $\A{t}$ of $X$. 
\end{definition}
The correctness of the above definition is a consequence of a simple observation that the contact lift preserves the composition of maps, i.e., $\proj{F\circ G}=\proj{F}\circ\proj{G}$ for any pair of maps $F,G:M\ra M$. It follows that the contact lift of the flow $\A{t}$ is a flow of contact transformations of $\proj{\T^\ast M}$ and as such it must correspond to some contact vector field (cf. Definition \ref{def:cvf}).

An analogous reasoning shows that given a \nice\ TDVF $X_t\in\X(M)$ and the related TD flow $\A{t\tau}:M\ra M$, the contact lift of the latter, i.e., $\PP(\A{t\tau})$, will consist of contact transformations and will satisfy all the properties of the TD flow. By the results of Proposition \ref{prop:contact_evolution}, $\proj{\A{t\tau}}$ is a TD flow associated with some contact TDVF (see also Lemma \ref{lem:exist_TDVF}). Obviously this field is just $\cont{X_t}$. The justification of this fact is left for the reader. 

\begin{lemma}
\label{lem:PX_contact} 
The CVF $\cont{X}$ is generated (in the sense of Lemma \ref{lem:cvf}) by the $\N\C$-class of $\wtilde X$, where $\wtilde X$ is any smooth vector field on $\PP(\T^\ast M)$ which projects to $X$ under $\T\pi$, i.e.,
$$\cont{X}=\sym{[\wtilde X]}\ .$$  
\end{lemma}  
\begin{proof}
Since $\proj{\A{t}}$, the flow of $\cont{X}$, projects under $\pi$ to $\A{t}$, the flow of $X$, we conclude that $X=\T\pi(\cont{X})$. As we already know from the proof of Lemma \ref{lem:cvf}, a CVF is uniquely determined by its class in $\N\C$. By \eqref{eqn:cont_tpm}, the $\N\C$-class of a field $Y\in\X(\PP(\T^\ast M))$ is completely determined by its $\T\pi$-projection. In other words, if two fields $Y$ and $Y'$ have the same $\T\pi$-projections, then $Y-Y'$ is a $\C$-valued vector field. Thus the field $\wtilde X$ has the same $\N\C$-class as the CVF $\cont{X}$ so, by the results of Lemma \ref{lem:cvf} (see also Remark \ref{rem:generator}), it follows $\cont{X}=\sym{[\wtilde X]}$.    
\end{proof}

\begin{remark}
\label{rem:PX_generator}
We shall end our considerations about the contact lift $\cont{X}$ by discussing its description in terms of generating functions (cf. our comments following Remark \ref{rem:generator}). Let us choose a vector field $R$ on $M$ and fix a contact form $\omega_R=(i_R)^\ast\Omega_M$ on $\U_R\subset\PP(\T^\ast M)$. Using the results of our  previous Section \ref{sec:contact} and with the help of the contact form $\omega_R$, the CVF $\cont{X}$ can be presented as $\sym{\phi}$ for some generating function $\phi:\U_R\ra\R$. This function is simply the evaluation of $\omega_R$ at $\cont{X}$. In fact,
$$\phi([\theta])=\omega_R\big|_{[\theta]}(\cont{X})=\omega_R\big|_{[\theta]}(\sym{[\wtilde X]})\overset{\mathrm{Rem.}\ \ref{rem:generator}}=\omega_R\big|_{[\theta]}(\wtilde X)\overset{\eqref{eqn:liouville}}=\<\theta,\T\pi(\wtilde X)>=\<\theta,X>\ ,$$
where $\theta=i_R([\theta])$, i.e. $\theta(R)=1$. In other words, the value of the   generating function of $\cont{X}$ on the class $[\theta]$ equals the value of the representative $\theta$, defined by $\theta(R)=1$, on the field $X$ which is being lifted. 
\end{remark}

\begin{remark}
It is worth mentioning the following illustrative picture which was pointed to us by Janusz Grabowski. Every contact structure on a manifold $N$ can be viewed as a homogeneous symplectic structure on some principal $GL(1,\R)$-bundle over $N$. In the case of the canonical contact structure on $N=\PP(\T^\ast M)$ the corresponding bundle is simply $\T^\ast_0 M$, the cotangent bundle of $M$ with the zero section removed, equipped with the natural action of $\R\setminus\{0\}=GL(1,\R)$ being the restriction of the multiplication by reals on $\T^\ast M$. The canonical symplectic structure is obviously homogeneous with respect to this action. Now every homogeneous symplectic dynamics on $\T_0^\ast M$ reduce to contact dynamics on $\PP(\T^\ast M)$. For more information on this approach the reader should consult \cite{Grabowski_2013} and \cite{Bruce_Grabowska_Grabowski_2015}.
\end{remark}
\section{The Pontryagin Maximum Principle}
\label{sec:pmp}

\paragraph{The Pontryagin Maximum Principle.}
A \emph{control system} on a manifold $Q$ is constituted by a family of vector fields $f:Q\times U\ra \T Q$ parametrized by a topological space $U$. It can be understood as a parameter-dependent differential equation
\begin{equation}
\label{eqn:cs}\tag{CS}
\dot q(t)= f(q(t),u(t)),\quad u(t)\in U\ .
\end{equation}
For a given measurable and locally bounded $u(t)\in U$, the solution $q(t)$ of \eqref{eqn:cs} is usually called a \emph{trajectory} of \eqref{eqn:cs} associated with the \emph{control} $u(t)$. 

An introduction of a \emph{cost function} $L:Q\times U\ra\R$ allows to consider the following \emph{optimal control problem} \eqref{eqn:ocp}
\begin{equation}
\label{eqn:ocp}\tag{OCP}
\begin{split}
&\dot q(t)= f(q(t),u(t)),\quad u(t)\in U\\
&\int_0^TL(q(t),u(t))\dd t \lra \min\ .
\end{split}
\end{equation}
The minimization is performed over $u(t)$'s which are locally bounded and measurable, the time interval $[0,T]$ is fixed and we are considering fixed-end-points boundary conditions $q(0)$ and $q(T)$. By a solution of the optimal control problem we shall understand a pair $(q(t),u(t))$ satisfying \eqref{eqn:ocp}. 

Let now $q(t)\in Q$ be the trajectory of the \eqref{eqn:cs} associated with a given control $u(t)\in U$. It is convenient to regard the related trajectory $\wt q(t)=(q(t),q_0(t))$ in the extended configuration space $\wt Q:=Q\times\R$, where $q_0(t):=\int_0^tL(q(s),u(s))\dd s$ is the cost of the trajectory at time $t$.\footnote{From now on, geometric objects and constructions associated with the extended configuration space $\wt Q$ will be denoted by bold symbols, e.g., $\wt f$, $\wt q$, $\Flow{tt_0}$, $\wt \HH_t$ etc. Normal-font symbols, e.g., $f$, $q$, $\flow{tt_0}$, $\HH_t$, will denote analogous objects in $Q$ being, in general, projections of the corresponding objects from $\wt Q$.} In fact, $\wt q(t)$ is a trajectory (associated with the same control $u(t)$) of the following extension of \eqref{eqn:cs}:
\begin{equation}
\label{eqn:cs1}\tag{$\wt{\mathrm{CS}}$}
\dot{\wt q}(t)=\wt f(\wt q(t),u(t))\quad u(t)\in U \ ,
\end{equation}
with $\wt f:=(f,L\cdot\pa_{q_0}): \wt Q\times U\ra \T\wt Q=\T Q\times\T\R$. Here we treat both $f$ and $L$ as maps from $\wt Q\times U$ invariant in the $\R$-direction in $\wt Q=Q\times\R$. In other words, we extended \eqref{eqn:cs} by incorporating the costs $q_0(t)$ as additional configurations of the system. The evolution of these additional configurations is governed by the cost function $L$. Note that the total cost of the trajectory $q(t)$ with $t\in[0,T]$ is precisely $q_0(T)$. Since the latter is fully determined by the pair $(q(t),u(t))$, it is natural to regard the extended pair $(\wt q(t),u(t))$ rather than $(q(t),u(t))$ as a solution of \eqref{eqn:ocp}.

Note that the extended configuration space $\wt Q=Q\times\R\ni\wt q=(q,q_0)$ is equipped with the canonical vector field $\wt\pa_{q_0}:=(0,\pa_{q_0})\in\T\wt Q=\T Q\times\T\R$. We shall denote the distribution spanned by this field by $\wt \RR\subset\T\wt Q$. The ray $\wt \RR^{-}_{\wt q}:=\{-r\wt\pa_{q_0}\ |\ r\in\R_{+}\}\subset\wt \RR_{\wt q}\subset\T_{\wt q}\wt Q$ contained in this distribution will be called the \emph{direction of the decreasing cost at $\wt q\in\wt Q$}.

Regarding technical assumptions, following \cite{Pontr_Inn_math_theor_opt_proc_1962}, we shall assume that $U$ is a subset of an Euclidean space, $f(q,u)$ and $L(q,u)$ are differentiable with respect to the first variable and, moreover, $f(q,u)$, $L(q,u)$, $\frac{\pa f}{\pa q}(q,u)$ and $\frac{\pa L}{\pa q}(q,u)$ are continuous as functions of $(q,u)$. In the light of Theorem \ref{thm:solutions_ODE} below it is clear that these conditions guarantee that, for any choice of a bounded measurable control $u(t)$ and any initial condition $\wt q(0)$, equation \eqref{eqn:cs1} has a unique (Caratheodory) solution defined in a neighborhood of $0$. It will be convenient to denote the TDVF's $q\mapsto f(q,u(t))$ and $\bm q\mapsto \bm f(\bm q,u(t))$ related with such a control $u(t)$ by $f_{u(t)}$ and $\bm f_{u(t)}$, respectively. In the language of Section \ref{sec:technical} technical assumptions considered above guarantee that $f_{u(t)}$ and $\bm f_{u(t)}$ are \nice\ TDVF's. In particular their TD flows $\flow{t\tau}:Q\ra Q$ and $\Flow{t\tau}:\wt Q\ra\wt Q$, respectively, are well-defined families of (local) diffeomorphisms.\footnote{From now on we will use symbols $\Flow{t\tau}$ and $\flow{t\tau}$ to denote the TD flows of TDVF's $\bm f_{u(t)}$ and $f_{u(t)}$, respectively, for a particular control $u(t)$. Note that $\Flow{t\tau}$ projects to $\flow{t\tau}$ under $\pi_1:\bm Q=Q\times \R\ra Q$.} Note that if $\wt q(t)$ with $t\in[0,T]$ is a solution of \eqref{eqn:cs1}, then for every $t,\tau\in[0,T]$ the map $\Flow{t\tau}$ is well-defined in a neighborhood of $\wt q(\tau)$.

In the above setting necessary conditions for the optimality of $(\wt q(t),u(t))$ are formulated in the following \emph{Pontryagin Maximum Principle} (PMP, in short)
\begin{theorem}[\cite{Pontr_Inn_math_theor_opt_proc_1962}]
\label{thm:pmp_hamiltonian}
Let $(\wt q(t),\wh u(t))$ be a solution of the \eqref{eqn:ocp}. Then for each $t\in [0,T]$ there exists a non-zero covector $\wt \lambda(t)\in \T^\ast_{\wt q(t)}\wt Q$ such that the curve $\wt \Lambda_t=(\wt q(t),\wt \lambda(t))$ satisfies the time-dependent Hamilton equation
\begin{equation}
\label{eqn:pmp_evolution_ham}
\dot{\wt \Lambda_t}=\vec{\wt H_t}(\wt \Lambda_t)\ , 
\end{equation}
where $\vec{\wt H_t}$ denotes the Hamiltonian vector field on $\T^\ast \wt Q$ corresponding to the time-dependent Hamiltonian
\begin{equation}
\label{eqn:hamiltonian}
\wt H_t(\wt q,\wt \lambda):=\<\wt \lambda,\wt f_{\wh u(t)}(\wt q)>\ .
\end{equation}
Moreover, along $\wt \Lambda_t$ the Hamiltonian $\wt H_t$ satisfies the following \emph{Maximum Principle}
\begin{equation}
\label{eqn:max_princip}
\wt H_t(\wt q(t),\wt \lambda(t))=\max_{v\in U}\<\wt \lambda(t),\wt f_{v}(\wt q(t))>\ .
\end{equation}
\end{theorem}

\begin{definition}
\label{def:extremal}
 A pair $(\wt q(t),\wh u(t))$ satisfying the necessary conditions for optimality provided by Theorem~\ref{thm:pmp_hamiltonian} (i.e., the existence of a covector curve $\wt \Lambda_t$ satisfying \eqref{eqn:pmp_evolution_ham}--\eqref{eqn:max_princip}) is called an \emph{extremal}.  
\end{definition}

\paragraph{Proof of the PMP.} Although the PMP is a commonly-known result, for future purposes it will be convenient to sketch  its original proof following \cite{Pontr_Inn_math_theor_opt_proc_1962}. 

Let $(\wt q(t), \wh u(t))$ be a trajectory of \eqref{eqn:cs1}. By $\Flow{t\tau}: \wt Q\ra\wt Q$, where $0\leq \tau\leq t\leq T$, denote the TD flow on $\wt Q$ of  the \nice\ TDVF $\wt f_{\wh u(t)}$ defined by the control $\wh u(t)$ (cf. Definition \ref{def:td_flow}). In other words, given a point $\wt q\in \wt Q$, the curve $t\mapsto \Flow{t 0}(\wt q)$ is the a trajectory of \eqref{eqn:cs1} associated with the control $\wh u(t)$ with the initial condition $\wt q(0)=\wt q$.

The crucial step in the proof of the PMP is introducing the, so called, \emph{needle variations} and the resulting construction of a family of sets\footnote{In the original proof in \cite{Pontr_Inn_math_theor_opt_proc_1962} the optimal control problem with a free time interval $[0,T]$ is considered. In this case, the sets $\wt\K_t$ contain additional elements. }
\begin{equation}
\label{eqn:cone}
\wt\K_t:=\operatorname{cl}\left\{\sum_{i=1}^k\Flow{t t_i}\left[\wt f_{v_i}(\wt q(t_i))-\wt f_{\wh u(t_i)}(\wt q(t_i))\right]\delta t_i\right\}\ ,
\end{equation}
where $0< t_1\leq t_2\leq\hdots \leq t_k\leq t<T$ is any finite sequence of regular points (see Appendix \ref{sec:appendix}) of the control $\wh u(\cdot)$,  $v_i$ are arbitrary elements in $U$ and $\delta t_i$ are arbitrary non-negative numbers. It is easy to see that $\wt \K_t$ is a closed and convex cone in $\T_{\wt q(t)}\wt Q$, well-defined for each regular point $t\in (0,T)$ of the control $\wh u(\cdot)$. What is more, the cones $\wt \K_t$ are ordered by the TD flow $\Flow{t\tau}$, i.e.,
\begin{equation}
\label{eqn:evolution_cone}
\T \Flow{t\tau}(\wt \K_\tau)\subset \wt \K_t\ ,
\end{equation}  
for each pair of regular points $0<\tau\leq t< T$. The above property allows to extend the construction of $\wt \K_t$ to non-regular points of $\wh u(\cdot)$ (including the end-point $T$) by setting
$$\wt \K_t:=\bigcup_{\{\tau\ |\ \tau\leq t\ \text{and $\tau$ regular}\}} \T\Flow{t\tau}(\wt \K_\tau)\subset \T_{\wt q(t)}\wt Q\ .$$
Clearly, these sets preserve all important features of $\wt \K_t$'s, i.e., they are closed and convex cones satisfying \eqref{eqn:evolution_cone} for any pair of points $0<\tau\leq t\leq T$.  

The importance of the construction of the cone $\wt \K_t$ lies in the fact that it approximates the reachable set of the control system \eqref{eqn:cs1} at the point $\wt q(t)$. In particular, it was proved in \cite{Pontr_Inn_math_theor_opt_proc_1962} that if at any point $t\in[0,T]$, the interior of the cone $\wt \K_t$ contains the direction of the decreasing cost $\wt \RR^-_{\wt q(t)}$, then the trajectory $t\mapsto\wt q(t)$, $t\in[0,T]$, cannot be optimal. 

\begin{lemma}[\cite{Pontr_Inn_math_theor_opt_proc_1962}]
If,  for any $0<t\leq T$, the ray $\wt \RR^-_{\wt q(t)}$ lies in the interior of $\wt\K_t$, then $(\wt q(t),\wh u(t))$ cannot be a solution of \eqref{eqn:ocp}. 
\end{lemma}
As a direct corollary, using basic facts about separation of convex sets,  one obtains the following 

\begin{proposition}[\cite{Pontr_Inn_math_theor_opt_proc_1962}]
Assume that $(\wt q(t),\wh u(t))$ is a solution of \eqref{eqn:ocp}. Then for each $t\in(0,T]$ there exists a hyperplane $\wt \HH_t\subset \T_{\wt q(t)}\wt Q$ separating the convex cone $\wt\K_t$ from the ray $\wt\RR^-_{\wt q(t)}$.   
\end{proposition}
Thus one can choose a curve of hyperplanes\footnote{By choosing $\wt\K_0:=\{0\}$ we can easily extend $\wt\HH_t$ to the whole interval $[0,T]$.} $\wt\HH_t\subset\T_{\wt q(t)}\wt Q$ separating the cone $\wt\K_t$ from the ray $\wt\RR^-_{\wt q(t)}$ for each $t\in (0,T]$. Because of \eqref{eqn:evolution_cone} and the fact that the canonical vector field $\wt\pa_{q_0}$ is invariant under $\T \Flow{t\tau}$ (the control does not depend on the cost), we may choose $\wt\HH_t$ in such a way that 
\begin{equation}
\label{eqn:evolution_H}
\T \Flow{t\tau}(\wt \HH_\tau)=\wt \HH_t\ ,
\end{equation}
for each $0\leq \tau\leq t\leq T$. Indeed, the basic idea is to choose any $\wt\HH_T$ separating $\wt\K_T$ from $\wt\RR^{-}_{\wt q(T)}$ and define $\wt\HH_t:=\T \Flow{Tt}^{-1}(\wt\HH_T)$ for every $t\in[0,T]$. We leave the reader to check that such a construction has the desired properties. 

The geometry of this situation is depicted in Figure \ref{fig:pmp}.
\begin{figure}[ht]%
\begin{center}
\def\svgwidth{0.9 \columnwidth}
\input{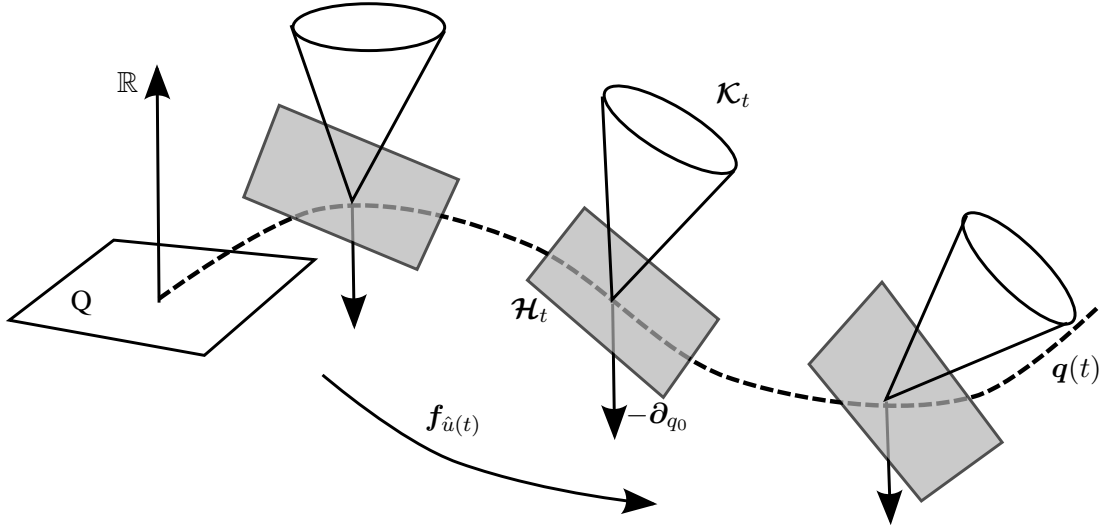}
\caption{Geometrically the PMP describes a family of cones $\wt\K_t$ along the optimal solution $\wt q(t)$ separated from the direction of the decreasing cost $\wt\RR^-_{\wt q(t)}$ by hyperplanes $\wt\HH_t$. Both $\wt\K_t$ and $\wt\HH_t$ evolve according to the extremal vector field $\wt f_{\wh u(t)}$.}
\label{fig:pmp}%
\end{center}
\end{figure}

\begin{remark}
\label{rem:abnormal}
Trajectories of \eqref{eqn:cs1} satisfying the above necessary conditions for optimality (i.e., the existence of a curve of separating hyperplanes $\wt\HH_t$ which satisfies \eqref{eqn:evolution_H}) can be classified according to the relative position of the hyperplanes $\wt\HH_t$ and the line field $\wt\RR\subset\T\wt Q$. Note that, since the hyperplanes $\wt\HH_t$ evolve according to the TD flow $\Flow{t\tau}$ of the TDVF $\wt f_{\wh u(t)}$, which leaves the distribution $\wt\RR$ invariant, we conclude that whenever $\wt\RR_{\wt q(\tau)}\subset\wt\HH_\tau$ at a particular point $\tau\in[0,T]$, then $\wt\RR_{\wt q(t)}\subset\wt\HH_t$ for every $t\in[0,T]$. We call a trajectory $\wt q(t)$ of \eqref{eqn:cs1} satisfying the above necessary conditions for optimality:
\begin{itemize}
	\item \emph{normal} if $\wt\RR_{\wt q(t)}\not\subset\wt\HH_t$ for any $t\in[0,T]$. Note that, in consequence, the ray $\wt\RR^{-}_{\wt q(t)}$ can be strictly separated from the cone $\wt\K_t$ for each $t\in[0,T]$. 
	\item \emph{abnormal} if $\wt\RR_{\wt q(t)}\subset\wt\HH_t$ for each $t\in[0,T]$.
	\item \emph{strictly abnormal} if for some $t\in[0,T]$ the ray $\wt\RR^-_{\wt q(t)}$ cannot be strictly separated from the cone $\wt\K_t$ (and thus $\bm\RR_{\bm q(t)}\subset\bm\HH_t$ for each $t\in[0,T]$). 
\end{itemize}

It is worth emphasizing that being normal or abnormal is not a property of a trajectory itself, but of a trajectory together with a particular curve of separating hyperplanes. Thus, \emph{a priori}, a given trajectory $\wt q(t)$ may admit two different curves of separating hyperplanes, one being normal and the other abnormal. On the other hand, if $\wt q(t)$ is a strictly abnormal trajectory it must be abnormal (and cannot be normal) for any possible choice of the curve of separating hyperplanes. To justify this  statement observe that if the ray $\wt \RR^-_{\wt q(t)}$ cannot be strictly separated from the cone $\wt\K_t$, then necessarily (since the cones $\wt\K_t$ are closed) we have $-\wt\pa_{q_0}\in\wt\K_t$ for some $t\in(0,T]$. Consequently, also $-\wt\pa_{q_0}\in \wt\HH_t$, as $\wt\HH_t$ separates $-\wt\pa_{q_0}\in\wt\K_t$ and $-\wt\pa_{q_0}\in \wt\RR^-_{\wt q(t)}$ (see Figure \ref{fig:abnormal}). Note that since $\wt\HH_t$ is a linear space, the whole line $\wt\RR_{\wt q(t)}$ spanned by the vector $-\wt\pa_{q_0}$ is contained in $\wt\HH_t$ in this case.  
 
\begin{figure}[ht]%
\begin{center}
\def\svgwidth{0.25 \columnwidth}
\input{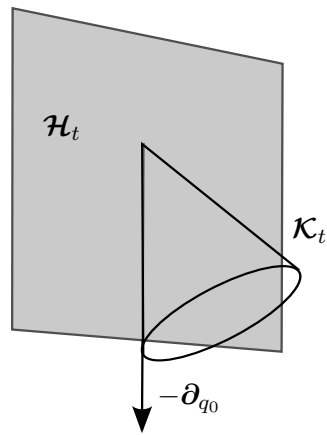}
\caption{For strictly abnormal trajectories the cone $\wt\K_t$ cannot be strictly separated from the direction of the decreasing cost $\wt\RR^-_{\wt q(t)}$ for some $t\in[0,T]$. Consequently, $\wt\RR_{\wt q(t)}\subset \wt\HH_t$ for each $t\in[0,T]$.}
\label{fig:abnormal}%
\end{center}
\end{figure}
\end{remark}

It is precisely only now when the covector $\wt\lambda(t)$ of the PMP appears. Namely, one can represent each hyperplane $\wt\HH_t\subset\T_{\wt q(t)}\wt Q$ as the kernel of a covector $\wt\lambda(t)\in\T^\ast_{\wt q(t)}\wt Q$. Due to \eqref{eqn:evolution_H} it is possible to choose these covectors in such a way that for every $0\leq \tau\leq t\leq T$ the curve $\wt\Lambda_t=(\wt q(t),\wt\lambda(t))$ satisfies 
$$\wt\Lambda_t=\left(\Flow{t\tau}^{-1}\right)^\ast \wt\Lambda_\tau\ .$$
This reads as the Hamilton equation \eqref{eqn:pmp_evolution_ham} in Theorem \ref{thm:pmp_hamiltonian}. Finally, the Maximum Principle \eqref{eqn:max_princip} follows from  the fact that $\wt\Lambda_t$ can be chosen in such a way that $\<\wt\lambda(t),-\wt\pa_{q_0}>\geq0\geq \<\wt\lambda(t),\wt k>$ for any $\wt k\in \wt\K_t$. The latter inequality for $\wt k=\wt f_v(\wt q(t))-\wt f_{\wh u(t)}(\wt q(t))$ regarded for each $v\in U$ implies \eqref{eqn:max_princip}. 

Note that, as we have already observed in Remark \ref{rem:abnormal}, for abnormal solutions, we have $\wt\pa_{q_0}\in \wt\HH_t=\ker\wt\lambda(t)$, and thus $\<\wt\lambda(t),\wt\pa_{q_0}>\equiv 0$. For normal solutions it is possible to choose $\wt\lambda(t)$ in such a way that $\<\wt\lambda(t),-\wt\pa_{q_0}>\equiv 1$ along the optimal solution.

\paragraph{The contact formulation of the PMP.}
Expressing the essential geometric information of the PMP (see Figure \ref{fig:pmp}) in terms of hyperplanes $\wt\HH_t$, instead of covectors $\wt\lambda(t)$, combined with our considerations about the canonical contact structure on $\PP(\T^\ast M)$ (see Section \ref{sec:ptm}) allows to formulate the following contact version of the PMP. 

\begin{theorem}[the PMP, a contact version]
\label{thm:pmp_contact}
Let $(\wt q(t),\wh u(t))$ be a solution of the \eqref{eqn:ocp}. Then for each $t\in [0,T]$ there exists a hyperplane $\wt\HH_t\in \PP(\T^\ast_{\wt q(t)}\wt Q)$ such that the curve $t\mapsto \wt\HH_t$ satisfies the equation
\begin{equation}
\label{eqn:pmp_evolution_contact}
\dot {\wt\HH_t}=\cont{\wt f_{\wh u(t)}}(\wt\HH_t)\ , 
\end{equation}
where $\cont{\wt f_{\wh u(t)}}$ denotes the contact TDVF on $\PP(\T^\ast \wt Q)$, being the contact lift of the TDVF $\wt f_{\wh u(t)}$ on $\wt Q$ (see Definition \ref{def:PX} and Lemma \ref{lem:PX_contact}). 

Moreover, each  $\wt\HH_t$ separates the convex cone $\wt\K_t$ defined by \eqref{eqn:cone} from the ray $\wt\RR^{-}_{\wt q(t)}$. 
\end{theorem}

\begin{proof}
The family of hyperplanes $\wt\HH_t$ separating the cone $\wt\K_t$ from the ray $\wt\RR^{-}_{\wt q(t)}$ and satisfying \eqref{eqn:evolution_H} was constructed in the course of the proof of Theorem \ref{thm:pmp_hamiltonian} sketched in the previous paragraph. To end the proof it is enough to check that $\wt\HH_t$ evolves according to \eqref{eqn:pmp_evolution_contact}. From \eqref{eqn:evolution_H} and Definition \ref{def:PF} of the contact lift  we know that $\wt\HH_t$ evolves according to $\PP(\Flow{\tau t})$. By the remark following Definition \ref{def:PX} this is precisely the TD flow induced by the TDVF $\cont{\wt f_{\wh u(t)}}$.    
\end{proof}

Let us remark that the contact dynamics \eqref{eqn:pmp_evolution_contact} are valid regardless of the fact whether the considered solution is normal or abnormal. We have a unique contact TDVF $\cont{\wt f_{\wh u(t)}}$ on $\PP(\T^\ast \wt Q)$ governing the dynamics of the separating hyperplanes $\wt\HH_t$. The difference between normal and abnormal solutions lies in the relative position of the hyperplanes $\wt\HH_t$ with respect to the canonical vector field $-\wt\pa_{q_0}$ on $\wt Q$. 

\begin{remark}
\label{rem:pmp_covariant}
Actually, the fact that the evolution of $\wt\HH_t$ is contact (and at the same time that the evolution of $\wt\Lambda_t$ is Hamiltonian) is in a sense ``accidental''. Namely, it is merely a natural contact (Hamiltonian) evolution induced on $\PP(\T^\ast\wt Q)$ (on $\T^\ast \wt Q$) by the TD flow on $\wt Q$ defined by means of the extremal vector field. In the Hamiltonian case this was, of course, already observed -- see, e.g., Chapter 12 in \cite{Agrachev_Sachkov_2004}. Thus it is perhaps more proper to speak rather about \emph{covariant} (in terms of hyperplanes) and \emph{contravariant} (in terms of covectors) \emph{formulations of the PMP}, than about its contact and Hamiltonian versions. It may seem that the choice between one of these two approaches is a matter of a personal taste, yet obviously the covariant formulation is closer to the original geometric meaning of the PMP, as it contains a direct information about the separating hyperplanes, contrary to the contravariant version where this information is translated to the language of covectors (not to forget the non-uniqueness of the choice of $\wt\Lambda_t$). In the next Section \ref{sec:appl} we shall show a few applications of the covariant approach to the sub-Riemannian geometry. Expressing the optimality in the language of hyperplanes allows to see a direct relation between abnormal extremals and special directions in the constraint distribution. It also provides an elegant geometric characterization of normal extremals.    
\end{remark}

Although equation \eqref{eqn:pmp_evolution_contact} has a very clear geometric interpretation it is more convenient to avoid, in applications, calculating the contact lift. Combining \eqref{eqn:evolution_H} with Theorem \ref{thm:flow_bracket} allows to substitute equation \eqref{eqn:pmp_evolution_contact} by a simple condition involving the Lie bracket.

\begin{theorem}[the PMP, a covariant version]
\label{thm:pmp_covariant}
Let $(\wt q(t),\wh u(t))$ be a solution of \eqref{eqn:ocp}. Then for each $t\in [0,T]$ there exists a hyperplane $\wt\HH_t\in \PP(\T^\ast_{\wt q(t)}\wt Q)$ such that the curve $t\mapsto \wt\HH_t$ satisfies the equation \eqref{eqn:evolution_H}. Equivalently, $\wt\HH_t$ is a curve of hyperplanes which is a \sexy\ distribution that is $\wt f_{\wh u(t)}$-invariant along $\bm q(t)$, i.e., 
\begin{equation}
\label{eqn:pmp_bracket}
[\wt f_{\wh u(t)},\wt\HH_t]_{\bm q(t)}\subset\wt\HH_t\quad \text{for a.e. $t\in[0,T]$}\ . 
\end{equation}

Moreover, each  $\wt\HH_t$ separates the convex cone $\wt\K_t$ defined by \eqref{eqn:cone} from the ray $\wt\RR^{-}_{\wt q(t)}$. 
\end{theorem}

\begin{proof}
The proof is immediate. The existence of separating hyperplanes $\bm \HH_t$ satisfying \eqref{eqn:evolution_H} was already proved in the course of this section. The only part that needs some attention is the justification of equation \eqref{eqn:pmp_bracket}. It follows directly from the $\Flow{t\tau}$-invariance along $\bm q(t)$ of $\bm \HH_t$ and Theorem \ref{thm:flow_bracket}. (Note that $\bm \HH_t$ is \sexy\ along $\bm q(t)$ by Proposition \ref{prop:sexy_examples}.) 
\end{proof}

Finally, let us discuss the description of the contact dynamics \eqref{eqn:pmp_evolution_contact} in terms of natural contact forms introduced in the proof of Lemma \ref{lem:PTM_contact}. Recall that a choice of a vector field $\wt R\in \X(\wt Q)$ allows to define a natural embedding of the open set $\U_{\wt R}=\{[\wt \theta]\ |\ \wt\theta(\wt R)\neq 0\}\subset\PP(\T^\ast \wt Q)$ into $\T^\ast \wt Q$ (recall that in the language of hyperplanes, the set $\U_{\wt R}$ consists of those hyperplanes in $\T\wt Q$ which are transversal to the field $\wt R$). What is more, the pullback $\omega_{\wt R}$ of the Liouville form $\Omega_{\wt Q}$ on $\T^\ast\wt Q$, is a contact form on $\U_{\wt R}$. By the comment of Remark \ref{rem:PX_generator}, the generating function of the CVF $\cont{\wt f_{\wh u(t)}}$ associated with the contact form $\omega_{\wt R}$ is simply 
$$\PP(\T^\ast\wt Q)\supset\U_{\wt R}\ni (\wt q,[\wt\lambda])\longmapsto \<\wt\lambda, \wt f_{\wh u(t)}(\wt q) >\overset{\eqref{eqn:hamiltonian}}=\wt H_t(\wt q,\wt\lambda)\in\R\ ,$$
where $\wt\lambda$ is a representative of the class $[\wt\lambda]$ such that $\<\wt\lambda,\wt R>=1$.  In other words, the generating function of the contact dynamics \eqref{eqn:pmp_evolution_contact} associated with $\omega_{\wt R}$ is precisely the linear Hamiltonian \eqref{eqn:hamiltonian}. 
 
In particular, by choosing $\wt R=\wt\pa_{q_0}$ we can easily recover the results of \cite{Ohsawa_contact_pmp_2015}. Note that $\wt R=\wt\pa_{q_0}$ is the canonical choice of a vector field transversal to all hyperplanes $\wt\HH_t$'s in the normal case (note that additionally $\wt R=\wt\pa_{q_0}$ is $\Flow{t\tau}$-invariant). For such a choice of $\wt R$, the corresponding embedding $i_{\wt R}:\U_{\wt R}\hookrightarrow \T^\ast \wt Q$ is constructed simply by setting $\<\wt \lambda,\wt\pa_{q_0}>=1$, which is just the standard normalization of the normal solution. The associated contact form is $\omega_{\wt R}=\pi_2^\ast\dd q_0+\pi_1^\ast\Omega_Q$, where $\Omega_Q$ is the Liouville form on $\T^\ast Q$ and $\pi_1:Q\times \R\ra Q$ and $\pi_2:Q\times \R\ra \R$ are natural projections. As observed above, the generating function of the contact dynamics associated with $\omega_{\wt R}$ is the linear Hamiltonian \eqref{eqn:hamiltonian}.  
This stays in a perfect agreement with the results of Section 2 in \cite{Ohsawa_contact_pmp_2015}.

For the abnormal case there is no canonical choice of the field $\wt R$ transversal to the separating planes. Yet locally such a choice (but not canonical) is possible. The resulting generating function of the contact dynamics \eqref{eqn:pmp_evolution_contact} is again the linear Hamiltonian \eqref{eqn:hamiltonian}.

\section{Applications to the sub-Riemannian geodesic problem}
\label{sec:appl}

In this section we shall apply our covariant approach to the PMP (cf. Remark \ref{rem:pmp_covariant}) to concrete problems of optimal control. We shall concentrate our attention on the sub-Riemannian (SR, in short) geodesic problem on a manifold $Q$. Our main idea is to extract, from the geometry of the cone $\wt\K_t$, as much information as possible about the separating hyperplane $\wt\HH_t$ and then use the contact evolution (in the form \eqref{eqn:evolution_H} or \eqref{eqn:pmp_bracket}) to determine the actual extremals of the system. 

\paragraph{A sub-Riemannian geodesic problem.} To be more precise we are considering a control system constituted by choosing in the tangent space $\T Q$ a smooth constant-rank distribution $\D\subset \T Q$. Clearly (locally and non-canonically), by taking $f(q,u)=\sum_{i=1}^d u^if_i(q)$, where $u=(u^1,u^2,\hdots,u^d)$ and $\D=\<f_1,\hdots,f_d>$,  we may present $\D$ as the image of a map $f:Q\times U\ra \T Q$ where $U=\R^d$, with  $d:=\rank\D$, is an Euclidean space, i.e., a control system of type \eqref{eqn:cs}.  We shall refer to it as to the \emph{SR control system}. In agreement with our notation from the previous Section~\ref{sec:pmp} we will write also $f_u(q)$ instead of $f(q,u)\in \D_q$.  

The \emph{SR geodesic problem} is an optimal control problem of the form \eqref{eqn:ocp} constituted by considering a cost function 
$L(q,u):=\frac 12 g(f_u(q),f_u(q))$, where 
$$g:\D\times\D\ra\R$$
is a given symmetric positively-defined bi-linear form (\emph{SR metric}) on $\D$. 

In the considered situation, after passing to the extended configuration space $\wt Q=Q\times\R\ni(q,q_0)=\wt q$, the extended control function $\wt f:\wt Q \times U\ra \T\wt Q=\T Q\times\T\R$ is simply 
$$\wt f(\wt q,u)=\bm f_{u}(\bm q)=f_u(q)+\frac 12g(f_u(q),f_u(q))\pa_{q_0}\ .$$

\begin{definition}\label{def:SR_geodesic}
By a \emph{SR extremal} we shall understand a trajectory $(\wt q(t),\wh u(t))$ of the SR control system satisfying the necessary conditions for optimality given by the PMP (in the form provided by Theorem \ref{thm:pmp_hamiltonian} or, equivalently, Theorem \ref{thm:pmp_contact} or Theorem \ref{thm:pmp_covariant}).
\end{definition}

\paragraph{The geometry of cones and separating hyperplanes.}
Observe that the image $\wt f(\wt q,U)\subset\T_{\wt q}\wt Q=\T_qQ\times\T_{q_0}\R$ is a paraboloid (see Figure \ref{fig:paraboloid}). The following fact is a simple consequence of \eqref{eqn:cone}.

\begin{figure}[ht]%
\begin{center}
\def\svgwidth{0.6 \columnwidth}
\input{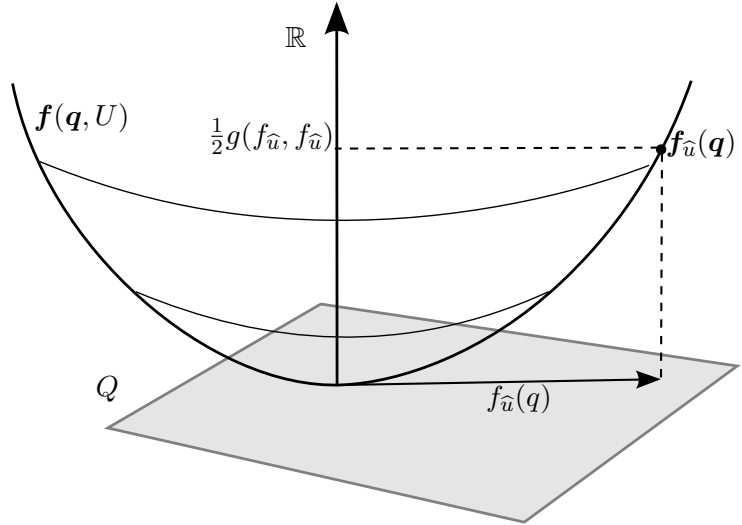}
\caption{In the sub-Riemannian case the image $\wt f(\wt q,U)$ is a paraboloid in $\T_{\wt q}\wt Q$.} 
\label{fig:paraboloid}
\end{center}
\end{figure}

\begin{lemma}
\label{lem:K_subspace}
Let $(\wt q(t),\wh u(t))$ be a trajectory of the SR control system and let $\wt\K_t$ be the associated  convex cone defined by formula \eqref{eqn:cone}. Then $\wt\K_t$ contains the tangent space of the paraboloid $\wt f(\wt q(t),U)$ at $\wt f_{\wh u(t)}(\wt q(t))$, i.e.,
\begin{equation}
\label{eqn:T_cone}
\{Y+g(f_{\wh u(t)},Y)\pa_{q_0}\ |\ Y\in \D_{q(t)}\}\subset\wt\K_t\ .
\end{equation}
\end{lemma}
\begin{proof}
It follows from \eqref{eqn:cone} (after taking $k=1$, $t_1=t$, and thus $\Flow{tt_1}=\id_{\wt Q}$) that $\wt\K_t$ contains every secant ray $\R_+\cdot\{\bm{f}_v(\bm q(t))- \bm {f}_{{\wh u}(t)}(\bm q(t))\}$ of the paraboloid $\wt f(\wt q(t),U)=\{\bm f_v(\wt q(t))\ |\ f_v(q(t))\in\D_{q(t)}\}$ passing through the point $\bm {f}_{{\wh u}(t)}(\bm q(t))$. Using these secant rays we may approximate every tangent ray of the paraboloid  $\wt f(\wt q(t),U)$ passing through $\bm {f}_{{\wh u}(t)}(\bm q(t))$ with an arbitrary accuracy. Since $\wt\K_t$ is closed, it has to contain this tangent ray and, consequently, the whole tangent space of $\wt f(\wt q(t),U)$ at $\bm {f}_{{\wh u}(t)}(\bm q(t))$ (see Figure \ref{fig:tangent}). The fact that this tangent space is described by equality \eqref{eqn:T_cone} is an easy exercise. 
\end{proof}

\begin{figure}[ht]%
\begin{center}
\def\svgwidth{0.6 \columnwidth}   
\input{tangent_fig.tex}
\caption{Since the cone $\wt\K_t$ contains all secant rays $\R_+\cdot\left\{\bm f_v({\wt q})- \bm f_{\wh u(t)}(\wt q)\right\}$ and is closed, it must contain the tangent space $\T_{\bm f_{\wh u(t)}(\wt q)}\wt f(\wt q,U)$.}
\label{fig:tangent}%
\end{center}
\end{figure}

\begin{remark}
\label{rem:tangen_gen_syst}
 In general, for an arbitrary control system and an arbitrary cost function, the cone $\wt \K_t$ contains all secant rays of the image $\wt f(\wt q(t),U)$ passing through $\wt f_{\wh u(t)}(\wt q(t))$. Thus, after passing to the limit, the whole tangent cone to $\wt f(\wt q(t),U)$ at $\wt f_{\wh u(t)}(\wt q(t))$ is contained in $\wt\K_t$. If $\wt f(\wt q(t),U)$ is a submanifold, as it is the case in the SR geodesic problem, this tangent cone is simply the tangent space at $\wt f_{\wh u(t)}(\wt q(t))$.
\end{remark}

Here is an easy corollary from the above lemma and our previous considerations. 
\begin{lemma}
\label{prop:dim_H}
Let $(\wt q(t), \wh u(t))$ be a SR extremal and let $\wt\HH_t\subset\T_{\wt q(t)}\wt Q$ be a curve of separating hyperplanes described in Theorem \ref{thm:pmp_covariant}. Then for each $t\in[0,T]$, the hyperplane $\wt\HH_t$ contains a $\rank\D$-dimensional linear subspace 
$$\{Y+g(f_{\wh u(t)},Y)\pa_{q_0}\ |\ Y\in \D_{q(t)}\}\subset\bm \HH_t\ .$$

If, additionally, $(\wt q(t), \wh u(t))$ is an abnormal SR extremal, then for each $t\in [0,T]$ there exists a hyperplane $\HH_t\subset\T_{q(t)} Q$ containing $\D_{q(t)}$, and such that the curve $t\mapsto \HH_t$ along $q(t)$ is subject to the evolution equation
\begin{equation}
\label{eqn:H_evolution_SR_Q_integral}
\T  \flow{t\tau}(\HH_\tau)= \HH_t\ ,\quad\text{for each $0\leq \tau\leq t\leq T$}.
\end{equation}
Here $\flow{t\tau}$ denotes the TD flow of a \nice\ TDVF $f_{\wh u(t)}$.   

Equivalently, $\HH_t$ is a curve of hyperplanes containing $\D_{q(t)}$ which is a \sexy\ distribution that is $f_{\wh u(t)}$-invariant  along $q(t)$, i.e., 
\begin{equation}
\label{eqn:H_evolution_SR_Q}
[f_{\wh u(t)},\HH_t]_{q(t)}\subset\HH_t \quad \text{for a.e. $t\in[0,T]$}, 
\end{equation}
\end{lemma}

\begin{proof}
To justify the first part of the assertion, observe that if, in a linear space $V$, a hyperplane $\HH\subset V$ supports a cone $\K\subset V$ which contains a line $l\subset\K$ (and all these sets contain the zero vector), then necessarily $l\subset \HH$ (each line containing 0 either intersects the hyperplane or is tangent to it). Since, by Lemma \ref{lem:K_subspace}, $\wt\K_t$ contains a subspace $\{Y+g(f_{\wh u(t)},Y)\pa_{q_0}\ |\ Y\in \D_{q(t)}\}$, we conclude that this subspace must lie in $\wt\HH_t$. 

Assume now that the considered extremal is abnormal. In this case, as we already observed in Remark~\ref{rem:abnormal}, $\wt\HH_t$ contains, in addition to the above-mentioned linear subspace, also the line $\wt\RR_{\wt q(t)}$ and thus we conclude that
$$\D_{q(t)}\oplus\RR_{q_0(t)}\subset\wt\HH_t\ .$$
Since $\{0_q\}\oplus \RR_{q_0}$ is the kernel of the natural projection $\T\pi_1:\T\bm Q\ra \T Q$, we conclude that, for every $t\in[0,T]$, the image $\HH_t$ of $\wt\HH_t$ under this projection is a hyperplane in $\T_{q(t)}Q$ which contains $\D_{q(t)}$. Obviously, since $\bm f_{\wh u(t)}$ projects to $f_{\wh u(t)}$,  equation \eqref{eqn:evolution_H} implies \eqref{eqn:H_evolution_SR_Q_integral}. By Theorem \ref{thm:flow_bracket}, equation \eqref{eqn:H_evolution_SR_Q} is the infinitesimal form of the latter. 
\end{proof}

It turns out that in  some cases the above basic information, suffices to find SR extremals. Let us study the following two examples. 

\begin{example}[Riemannian extremals]
\label{ex:geod}
In the Riemannian case $\D=\T Q$ is the full tangent space and $g$ is a Riemannian metric on $Q$. Let us introduce any connection $\nabla$ on $Q$ compatible with the metric. By $T_\nabla(X,Y):=\nabla_XY-\nabla_YX-[X,Y]$ denote the torsion of $\nabla$ (in particular, if we take the Levi-Civita connection $\nabla=\nabla^{LC}$, then $T_{\nabla^{LC}}\equiv 0$).

In this case $\rank\D=\dim Q$ so a Riemannian extremal cannot be abnormal from purely dimensional reasons: by Lemma \ref{prop:dim_H} in such a case a $(\dim Q-1)$-dimensional hyperplane  $\HH_t\subset\T_{q(t)}Q$ would contain a bigger $(\dim Q)$-dimensional space $\D_{q(t)}=\T_{q(t)}Q$, which is impossible. Thus every Riemannian extremal must be normal and, by Lemma \ref{prop:dim_H}, necessarily
$$\wt\HH_t=\{Y+g(f_{\wh u(t)},Y)\pa_{q_0}\ | \ Y\in\T_{q(t)}Q\}\ .$$

Now any $\wt\HH_t$-valued vector field along $\wt q(t)$ takes the form $Y(t)+g(f_{\wh u(t)},Y(t))\pa_{q_0}$ where $Y(t)\in\T_{q(t)}Q$. Its Lie bracket with the extremal vector field $\bm f_{\wh u(t)}=f_{\wh u(t)}+\frac 12g(f_{\wh u(t)},f_{\wh u(t)})\pa_{q_0}$ is simply (in the derivations we use the property that $Xg(Y,Z)=g(\nabla_XY,Z)+g(Y,\nabla_XZ)$ for every metric-compatible connection)
\begin{align*}
&\left[f_{\wh u(t)}+\frac 12g(f_{\wh u(t)},f_{\wh u(t)})\pa_{q_0},Y+g(f_{\wh u(t)},Y)\pa_{q_0}\right]=\\
&[f_{\wh u(t)},Y]+\left\{f_{\wh u(t)}g(f_{\wh u(t)},Y)-\frac 12Yg(f_{\wh u(t)},f_{\wh u(t)})\right\}\pa_{q_0}=\\
&[f_{\wh u(t)},Y]+\Big\{g(\nabla_{f_{\wh u(t)}}f_{\wh u(t)},Y)+g(f_{\wh u(t)},\nabla_{f_{\wh u(t)}}Y)-g(f_{\wh u(t)},\nabla_Y f_{\wh u(t)})\Big\}\pa_{q_0}=\\
&[f_{\wh u(t)},Y]+g(f_{\wh u(t)},[f_{\wh u(t)}, Y])\pa_{q_0}+\Big\{g(\nabla_{f_{\wh u(t)}}f_{\wh u(t)},Y)+g(f_{\wh u(t)},T_\nabla(f_{\wh u(t)}, Y))\Big\}\pa_{q_0}\ .
\end{align*} 
By \eqref{eqn:pmp_bracket} this Lie bracket should be $\wt\HH_t$-valued, and 
since $[f_{\wh u(t)},Y]+g(f_{\wh u(t)},[f_{\wh u(t)}, Y])\pa_{q_0}$ belongs to $\wt\HH_t$, we conclude that the considered bracket belongs to $\wt\HH_t$ if and only if for any $Y\in\T_{q(t)}Q$ we have
$$g(\nabla_{f_{\wh u(t)}}f_{\wh u(t)},Y)+g(f_{\wh u(t)},T_\nabla(f_{\wh u(t)}, Y))=0\ .$$
In this way we have expressed the geodesic equation for the metric $g$ in terms of the chosen metric-compatible connection $\nabla$ with torsion $T_\nabla$. In case that $\nabla=\nabla^{LC}$ is the Levi-Civita connection, the torsion vanishes and we recover the standard geodesic equation 
$$\overset{LC}\nabla_{f_{\wh u(t)}}f_{\wh u(t)}=0\ .$$ 
\end{example}

\begin{example}
Consider an abnormal SR extremal $(\wt q(t),\wh u(t))$ in a particular case of the SR geodesic problem  where $\D\subset\T Q$ is a co-rank one distribution. By Lemma \ref{prop:dim_H} in such a situation necessarily $\HH_t=\D_{q(t)}$, since the latter space is already of co-dimension one in $\T_{q(t)}Q$. Now \eqref{eqn:H_evolution_SR_Q} gives us
$$[f_{\wh u(t)},\D_{q(t)}]_{q(t)}\subset \D_{q(t)}$$
for almost every $t\in[0,T]$, i.e., in the considered case any abnormal extremal has to be a characteristic curve of $\D$. The converse statement is also true. Indeed, the reader may check that in this case $\bm \HH_t:=\D_{q(t)}\oplus\RR_{q_0(t)}$ is the curve of separating hyperplanes containing $\bm\RR_{\bm q(t)}$ and satisfying the assertion of Theorem \ref{thm:pmp_covariant} (see also the proof of Theorem \ref{thm:abnormal}).  
\end{example}

\noindent In the following two subsections we shall discuss normal and abnormal SR extremals in full generality. 

\subsection{Abnormal SR extremals}
Our previous considerations allow us to give the following characterization of SR abnormal extremals.

\begin{theorem}
\label{thm:abnormal}
For the SR geodesic problem introduced above the following conditions are equivalent:
\begin{enumerate}[(a)]
	\item\label{cond:A_abnormal} The pair $(\wt q(t),\wh u(t))$ is an abnormal SR extremal.
	\item\label{cond:B_abnormal} The smallest distribution $\flow{t\tau}$-invariant along $q(t)$ and containing $\D_{q(t)}$, i.e.,
	$$\flow{\bullet}(\D)_{q(t)}=\vect\{\T \flow{t\tau}(Y)\ |\ Y\in\D_{q(\tau)},\quad 0\leq \tau\leq T\}$$
	is of rank smaller than $\dim Q$. Here $\flow{t\tau}$ denotes the TD flow (in $Q$) of the \nice\ TDVF $f_{\wh u(t)}$.
\end{enumerate}

Moreover, condition \eqref{cond:B_abnormal} depends only on $f_{\wh u(t)}$ and $\D$ along $q(t)$. 
\end{theorem}
Note that Theorem \ref{thm:abnormal} reduces the problem of finding abnormal SR extremals to the study of the minimal distribution $\flow{t\tau}$-invariant along $q(t)$ and containing $\D_{q(t)}$. Often, if $q(t)$ is sufficiently regular, this problem can be solved by the methods introduced in Lemma \ref{cor:algorithm}, which are more practical from the computational view-point. 

\begin{corollary}
\label{cor:abnormal_smooth}
Let $X$ be a $C^\infty$-smooth $\D$-valued vector field and let $q(t)$ with $t\in[0,T]$ be an integral curve of $X$. Then $q(t)$ is a SR abnormal extremal in the following two (non-exhaustive) situations:
\begin{itemize}
	\item The distribution spanned by the iterated Lie brackets of $X$ with all possible smooth $\D$-valued vector fields, i.e.,
		$$\ad^\infty_X(\D)=\<\ad_X^k(Y)\ |\ Y\in \Sec(\D),\quad k=0,1,2,\hdots>$$
is of constant rank $r$ along $q(t)$ and  $r<\dim Q$.
\item  There exists a smooth  distribution $\B\supset \D$ on $Q$ of constant co-rank at least one, such that 
$$[X,\B]_{q(t)}\subset\B_{q(t)}\quad \text{for any $t\in[0,T]$.}$$
\end{itemize}

\end{corollary}
The above fact follows directly from Theorem \ref{thm:abnormal}, Lemma \ref{cor:algorithm} and Theorem \ref{thm:flow_bracket}. In each of the two cases along $q(t)$ we have a constant rank smooth (and thus \sexy) distribution  which contains $\D$, is $X$-invariant along $q(t)$ (and thus by Theorem \ref{thm:flow_bracket} also $\flow{t\tau}$-invariant along $q(t)$) and of co-rank at least one. Clearly such a distribution must contain $\flow\bullet(\D)_{q(t)}$, which in consequence also is of co-rank at least one. 

\begin{remark}
\label{rem:zhit_wrong}
In Sec. 7.3 in \cite{Zhitomirskii_1995} Zhitomirskii considered the following 2-distribution on $\R^5$ 
$$\D=\<X=\pa_x,Y=\pa_{y^1}+x\pa_{y^2}+(xy^1+h_1(x))\pa_{y^3}+(x(y^1)^2+h_2(x))\pa_{y_4})>\ ,$$ 
where $(x,y^1,y^2,y^3,y^4)$ are coordinates on $\R^5$ and smooth functions $h_1$ and $h_2$ satisfy the conditions 
$$\begin{cases}
h_1(x)=0 &\text{for $x\leq 0$}\\
h_1(x)\neq 0,\ h_1'(x)\neq 0,\ h_1''(x)\neq 0 &\text{for $x>0$}
\end{cases}$$
and
$$
\begin{cases}
h_2(x)=0 &\text{for $x\geq 0$}\\
h_2(x)\neq 0,\ h_2'(x)\neq 0,\ h_2''(x)\neq 0 &\text{for $x<0$}\ .
\end{cases}$$
Zhitomirskii proved that the curve $(-\varepsilon,\varepsilon)\ni t\mapsto (t,0,0,0,0)\in \R^5$ (which is obviously an integral curve of $X$) is not an abnormal SR extremal, yet, as he claims, the distribution $\ad^\infty_X(\D)$ regarded in the above corollary is of constant rank $r=4<5$ along this curve. A detailed study of this example reveals, however, that along the investigated curve, $r=4$ apart from the point $(0,0,0,0,0)$, where the rank drops down to 3. Thus the discussed example does not contradicts Corollary \ref{cor:abnormal_smooth}, as the regularity condition is not matched. In fact, the considered curve consists of two pieces of abnormal SR extremals (for $t>0$ and $t<0$) which do not concatenate to a single SR abnormal extremal, even though the concatenation is $C^\infty$-smooth. This example shows that the condition $r <\dim Q$ in Corollary \ref{cor:abnormal_smooth} is not sufficient (although it is necessary in the smooth case). 
\end{remark}

\begin{proof}[Proof of Theorem \ref{thm:abnormal}.]
If $(\wt q(t),\wh u(t))$ is an abnormal extremal then, by the results of Lemma \ref{prop:dim_H}, $\HH_t$, the $\T Q$-projection of the curve of supporting hyperplanes $\wt \HH_t\subset\T_{\wt q(t)} \wt Q$, is a curve of hyperplanes along $q(t)$ (i.e., a distribution of co-rank one along $q(t)$), it contains $\D_{q(t)}$ and is $\flow{t\tau}$-invariant along $q(t)$. In particular, it must contain the smallest distribution $\flow{t\tau}$-invariant  along $q(t)$ and containing $\D$ (cf. Proposition \ref{prop:AB}). Thus $\rank \flow\bullet(\D)_{q(t)}\leq\rank \HH_t=\dim Q-1$. \smallskip

Conversely, assume that $\rank \flow\bullet(\D)_{q(t)}<\dim Q$. Now by adding (if necessary) to $\flow\bullet(\D)_{q(t)}$ several vector fields of the form $\flow{tt_0}(X)$ where $X\in T_{q(0)}Q$, we can extend $\flow\bullet(\D)_{q(t)}$ to $\HH_t$, a co-rank one distribution $\flow{t\tau}$-invariant along $q(t)$. Define now the curve of hyperplanes $\wt\HH_t:=\HH_t\oplus\RR_{q_0(t)}\subset\T_{\wt q(t)}\wt Q$. We claim that $\wt \HH_t$ is a curve of supporting hyperplanes described in the assertion of Theorem \ref{thm:pmp_contact}. Indeed, the $\Flow{t\tau}$-invariance of $\wt\HH_t$ should be clear, as on the product $\wt Q=Q\times\R$ the TD flow $\Flow{t\tau}$ takes the form $\Flow{t\tau}(q,q_0)=(\flow{t\tau}(q), B_{t\tau}(q_0))$, for some TD flow $B_{t\tau}$ on $\R$. Clearly, since $\HH_t$ is $\flow{t\tau}$-invariant along $q(t)$, the tangent map of $\Flow{t\tau}$ preserves $\wt\HH_t=\HH_t\oplus\RR_{q_0(t)}$. To prove that $\wt \HH_t$ indeed separates the cone $\wt \K_t$ from the direction of the decreasing cost $\wt \RR^-_{\wt q(t)}$ observe that any vector of the form $\bm f_{v}(\wt q(t))-\bm f_{\wh u(t)}(\wt q(t))$, where $f_{v}\in\D_{q(t)}$, lies in $\D_{q(t)}\oplus\RR_{q_0(t)}\subset\wt\HH_t$. Moreover, any vector of the form  $\T \Flow{t\tau}\left[\bm f_v(\wt q(\tau))-\bm f_{\wh u(t)}(\wt q(\tau))\right]$, where $f_{v}\in\D_{q(\tau)}$, lies in $\T \Flow{t\tau}(\D_{q(\tau)}\oplus\RR_{q_0(\tau)})\subset\T\Flow{t\tau}(\wt\HH_\tau)\subset\wt\HH_t$. Thus, the whole cone $\wt\K_t$ is contained in $\wt\HH_t$ (cf. formula \eqref{eqn:cone}). Since also $\wt \RR^-_{\wt q(t)}\subset\wt\RR_{\wt q(t)}\subset\wt\HH_t$, we conclude that indeed $\wt\HH_t$ separates $\wt\K_t$ from $\wt \RR^-_{\wt q(t)}$ (in a trivial way). \smallskip

Finally, to justify the last statement of the assertion we can use Theorem \ref{thm:flow_bracket} to express the $\flow{t\tau}$-invariance of $\B_{q(t)}:=\flow\bullet(\D)_{q(t)}$ along $q(t)$ as the $f_{\wh u(t)}$-invariance of the latter distribution, and then use Lemma \ref{lem:along_x} (for $\B_{q(t)}\supset\D_{q(t)}\ni f_{\wh u(t)}(q(t))$) to prove that this invariance depends on $f_{\wh u(t)}$ and $\flow\bullet(\D)_{q(t)}$ along $q(t)$ only. Now it is enough to check that $\flow\bullet(\D)_{q(t)}$ itself does not depend on a particular choice of the extension of $f_{\wh u(t)}(q(t))$ to a neighborhood of $q(t)$. Assume thus that $f_{\wh u'(t)}$ is another extension of $f_{\wh u(t)}(q(t))$, that $\flow{t\tau}'$ is the related TD flow, and that $\flow{\bullet}'(\D)_{q(t)}$ is the minimal distribution $\flow{t\tau}'$-invariant along $q(t)$ and containing $\D_{q(t)}$. Now repeating the reasoning from the proof of Lemma \ref{lem:along_x} we would get
$$[f_{\wh u(t)},\flow\bullet(\D)]_{q(t)}=[f_{\wh u'(t)},\flow\bullet(\D)]_{q(t)}\mod \D_{q(t)}\ .$$
Since $[f_{\wh u(t)},\flow\bullet(\D)]_{q(t)}\subset \flow\bullet(\D)_{q(t)}$ and $\D_{q(t)}\subset\flow\bullet(\D)_{q(t)}$, we get $[f_{\wh u'(t)},\flow\bullet(D)]_{q(t)}\subset \flow\bullet(D)_{q(t)}$ which, by Theorem \ref{thm:flow_bracket}, implies that $\flow\bullet(D)$ is respected by the TD flow $\flow{t\tau}'$ along $q(t)$. From the minimality of $\flow{\bullet}'(\D)_{q(t)}$ we conclude that $\flow{\bullet}'(\D)_{q(t)}\subset \flow\bullet(\D)_{q(t)}$. Yet, for intertwined $f_{\wh u(t)}$ and $f_{\wh u'(t)}$ we would get the opposite inclusion in an analogous manner. Thus $\flow\bullet(\D)_{q(t)}=\flow{\bullet}'(\D)_{q(t)}$, and so it does not depend on the choice of the extension of $f_{\wh u(t)}$. This ends the proof. 
\end{proof}

\paragraph{Examples.}

\begin{example}
\label{ex:234}
Let $\D\subset\T Q$ be a smooth rank-2 distribution with the growth vector $(2,3,4,\hdots)$. Let $Y$, $Z$ be a local basis of sections of $\D$. From the form of the growth vector we conclude that the fields $Y$, $Z$ and $[Y,Z]$ are linearly independent, while the distribution
$$\<Y,Z,[Y,Z],[Y,[Y,Z]],[Z,[Y,Z]]>$$
is of rank 4. Thus the fields $[Y,[Y,Z]]$ and $[Z,[Y,Z]]$ are linearly dependent relative to the  distribution  $\<Y,Z,[Y,Z]>$, i.e., there exist smooth functions $\phi,\psi:Q\ra\R$ such that 
\begin{equation}
\label{eqn:2_distrib}
\phi[Y,[Y,Z]]+\psi[Z,[Y,Z]]=0\mod \<Y,Z,[Y,Z]>\ .
\end{equation}
We claim that the integral curves of the line bundle $\<\phi Y+\psi Z>\subset\D$ are SR abnormal extremals (notice that $\phi Y+\psi Z\in\D$ is a characteristic vector field of $\D+[\D,\D]$). To prove this we shall use the results of Corollary \ref{cor:abnormal_smooth}. Indeed, it is easy to check, using \eqref{eqn:2_distrib}, that for $X=\phi Y+\psi Z$ the smallest distribution $\ad_X$-invariant and containing $\D$ is simply the 3-distribution $\<Y,Z,[Y,Z]>$. This agrees with the results of Prop. 11 in \cite{Liu_Sussmann_1995} and Sec. 9 of \cite{Zhitomirskii_1995}. 
\end{example}

\begin{example}[Zelenko] 
\label{ex:dubrov} 
The following example by Igor Zelenko \cite{Zelenko_2006} become know to us thanks to the lecture of Boris Doubrov. The interested reader may consult also \cite{Agrachev_Zelenko_2006,Doubrov_Zelenko_preprint}.

Consider a 5-dimensional manifold $M$ with a 2-dimensional distribution $\B\subset \T M$ of type $(2,3,5)$. That is, locally $\B$ is spanned by a pair of vector fields $X_1$ and $X_2$ such that 
$$X_1,\quad X_2,\quad X_3:=[X_1,X_2],\quad X_4:=[X_1,X_3]\quad\text{and}\quad X_5:=[X_2,X_3]$$
form a local basis of sections of $\T M$. Consider now the bundle $Q:=\PP(\B)\subset\PP(\T M)\ra M$ of lines in $\B$, being a 6-dimensional manifold and a $\PP^1\R$-bundle over $M$. Introduce an affine chart $[1:t]$ corresponding to the line $\R\cdot\{X_1+tX_2\}$ on fibers of $Q\ra M$ and define a 2-dimensional distribution $\D:=\<\pa_t,X_1+tX_2>$ on $Q$. Our goal is to find abnormal SR extremals for this distribution. We will use Corollary \ref{cor:abnormal_smooth} for this purpose. 

First let us show that the integral curves of $\pa_t$ are abnormal extremals. Indeed, it is easy to see that $[\pa_t,X_1+t X_2]=X_2$ and that $[\pa_t,X_2]=0$, i.e., the minimal distribution ${\pa_t}$-invariant and containing $\D$ is precisely $\<\pa_t,X_1,X_2>$. This distribution is of constant rank smaller than $6=\dim Q$, so by Corollary \ref{cor:abnormal_smooth}, indeed, the integral curves of $\pa_t$ are abnormal extremals.

It is more challenging to find a second family of abnormal extremals of $\D$. Let us look for such a family being the integral curves of the field $H=X_1+t X_2+F\pa_t$, where $F$ is some, \emph{a priori} unknown, function on $Q$. To calculate the minimal distribution $H$-invariant and containing $\D$ it is enough to consider iterated Lie brackets $\ad^k_H(\pa_t)$. Skipping some simple calculations one can show that the vector fields
$$\pa_t, \quad H,\quad \ad_H(\pa_t)=[H,\pa_t],\quad\ad^2_H(\pa_t)\quad\text{and}\quad \ad^3_H(\pa_t)$$
span a 5-dimensional distribution $\wtilde \D$ on $Q$. Denote $[X_i,X_j]:=\sum_{k=1}^5f^k_{ij} X_k$ for $i,j=1,\hdots,5$. Then the Lie bracket $\ad^4_H(\pa_t)$ belongs to $\wtilde\D$ if and only if
$$F=-f^5_{14}+(f^4_{14}-2f^5_{24})t+(2f^4_{24}-f^5_{25})t^2+f^4_{24}t^3\ .$$
In such a case, $\wtilde \D$ is a constant rank distribution containing $\D$ and closed under $\ad_H(\cdot)$ (i.e., $H$-invariant). Since $\rank\wtilde D=5< \dim Q$, by Corollary \ref{cor:abnormal_smooth} the integral curves of $H$ (for $F$ as above) are abnormal SR extremals related with $\D$.
\end{example}

\begin{example}[Strongly bracket generating distributions]
\label{ex:sbg}
Recall that a distribution $\D\subset\T Q$ is called \emph{strongly bracket generating} (SBG, in short) if for any $p\in Q$ and any $X\in\Sec(\D)$ non-vanishing at $p$ we have
$$\D_p+[X,\D]_p=\T_pQ\ .$$
In the light of Corollary \ref{cor:abnormal_smooth} it is clear that a SR geodesic problem related with a SBG distribution does not admit any abnormal SR extremal. 

In fact, the same conclusion holds for a weaker version of the SBG condition, i.e., it is enough to assume that 
$$\D_p+[X,\D]_p+[X,[X,\D]]_p+\hdots=\T_pQ$$
for any $X\in\Sec(\D)$ non-vanishing at $p$. 
\end{example}

\begin{example}[Submanifold]
Assume that $S\subset Q$ is a submanifold of co-dimension at least one and such that $\D\big|_S\subset\T S$. Then any ACB curve $t\mapsto q(t)$ tangent to $\D$ and contained in $S$ is an abnormal extremal. Indeed, in this case $\T_{q(t)}S$ is obviously a \sexy\ distribution $\flow{t\tau}$-invariant along $q(t)$ which contains $\D_{q(t)}$ and is of co-rank at least one in $\T_{q(t)}Q$. Thus
$$\flow\bullet(\D)_{q(t)}\subset \T_{q(t)}S$$
and, consequently, $\flow\bullet(\D)_{q(t)}$ is of rank smaller than $\dim Q$. By Theorem \ref{thm:abnormal}, $q(t)$ is an abnormal extremal. 
\end{example}

\begin{example}[Zhitomirskii]
\label{ex:zhit_nice} 
Let $\D$ be a 2-distribution on a manifold $Q$ such that $\D^2:=\D+[\D,\D]$ is of rank 3. In \cite{Zhitomirskii_1995} Zhitomirskii introduced the following definition. 

A distribution $\mathcal{Z}\subset\T Q$ of co-dimension 2 is called \emph{nice with respect to $\D$} if
\begin{itemize}
	\item $\mathcal{Z}$ is involutive
	\item for any $q\in Q$ we have $\D_q\not\subseteq \mathcal{Z}_q$
	\item $\rank (\D^2\cap\mathcal{Z})=2$. 
\end{itemize}

In this case the intersection $\mathcal{L}:=\D\cap\mathcal{Z}$ is a line distribution. We shall show that the integral curves of $\mathcal{L}$ are abnormal SR extremals. Indeed, observe that $\D^2=\D^2\cap \mathcal Z+\D$ and thus
$$\HH:=\mathcal Z+\D^2=\mathcal Z+\D$$
is a smooth co-rank-one distribution in $Q$. Clearly $\D\subset \HH$ and, what is more, given any section $X\in \Sec(\mathcal L)$ we have $[X,\HH]\subset\HH$. Indeed, take any $\HH$-valued vector field $Y$. Since  $\HH=\mathcal Z+\D$ we can decompose it (in a non-unique way) as $Y=Y_1+Y_2$ where $Y_1\in\Sec(\mathcal Z)$ and $Y_2\in \Sec(\D)$. Now $[X,Y]=[X,Y_1]+[X,Y_2]$. Clearly $[X,Y_1]\in\Sec(\mathcal Z)$, since $X$ and $Y_1$ are $\mathcal Z$-valued and $\mathcal Z$ is involutive. Moreover $[X,Y_2]\in \Sec(\D^2)$, as both $X$ and $Y_2$ are $\D$-valued. We conclude that $[X,Y]=[X,Y_1]+[X,Y_2]\in\Sec(\mathcal Z+\D^2)=\Sec(\HH)$. 

Now it should be clear that the smallest distribution containing $\D$ and invariant with respect to to the TD flow of $Y$ is contained in $\HH$, which is of co-rank one. Thus, by Theorem \ref{thm:abnormal}, the integral curves of $X$ are abnormal SR extremals.
\end{example}

\subsection{Normal SR extremals}\label{ssec:normal}

Observe first that the extremal vector field $f_{\wh u(t)}$ is normalized by $g(f_{\wh u(t)},f_{\wh u(t)})\equiv 1$ along every solution of the SR geodesic problem. Indeed,
this follows easily from the standard argument involving the Cauchy-Schwartz inequality. From now on we shall thus assume that the extremal vector field $f_{\wh u(t)}$ is normalized in a neighborhood of a considered trajectory $q(t)$. This assumption allows for an elegant geometric characterization of normal SR extremals in terms of the distribution
$$\D^\perp_{q(t)}:=\{Y\in\D_{q(t)}\ |\ g(Y,f_{\wh u(t)})=0\}$$  
consisting of those elements of $\D$ which  are $g$-orthogonal to $f_{\wh u(t)}$ along $q(t)$. Note that $\D^\perp_{q(t)}$ is a subdistribution of $\D$ along $q(t)$.

\begin{theorem}[\cite{Alcheikh_Orro_Pelletier_1997}]
\label{thm:normal}
Assume that the field $f_{\wh u(t)}$ is normalized, i.e., $g(f_{\wh u(t)},f_{\wh u(t)})\equiv 1$. Then, for the SR geodesic problem introduced above, the following are equivalent:
\begin{enumerate}[(a)]
	\item\label{cond:A_normal} The pair $(\wt q(t),\wh u(t))$ is a normal SR extremal.
	\item\label{cond:B_normal} The velocity $f_{\wh u(t)}(q(t))$ is of class ACB with respect to $t$, and the smallest distribution $\flow{t\tau}$-invariant along $q(t)$ and containing $\D^\perp_{q(t)}$, i.e.,
	$$\flow\bullet(\D^\perp)_{q(t)}=\vect\{\T \flow{t\tau}(Y)\ |\ Y\in\D_{q(\tau)},\quad g(Y,f_{\wh u(t)})=0,\quad 0\leq \tau\leq T\}$$
	does not contain $f_{\wh u(t)}(q(t))$ for any $t\in[0,T]$. Here $\flow{t\tau}$ denotes the TD flow (in $Q$) of the \nice\ TDVF $f_{\wh u(t)}$.
\end{enumerate}
\end{theorem}
Theorem 3.1 of \cite{Alcheikh_Orro_Pelletier_1997} contains a formulation of the above result equivalent to ours.

Again if $q(t)$ is sufficiently regular we can use the method introduced in Lemma \ref{cor:algorithm} to check condition~\eqref{cond:B_normal} in the above theorem. The result stated below can be easily derived from   Theorem \ref{thm:normal} using similar arguments as in the proof of Corollary \ref{cor:abnormal_smooth}. For the case $\rank \D=2$ it was proved as Theorem 6 in \cite{Liu_Sussmann_1995}.  
\begin{corollary}\label{cor:normal}
Let $X$ be a $C^\infty$-smooth $\D$-valued vector field and let $q(t)$ with $t\in[0,T]$ be an integral curve of $X$. Then $q(t)$ is a SR normal extremal in the following two (non-exhaustive) situations:
\begin{itemize}
	\item The distribution spanned by the iterated Lie brackets of $X$ and all possible smooth $\D$-valued vector fields $g$-orthogonal to X, i.e.,
		$$\ad^\infty_X(\D^\perp)=\<\ad_X^k(Y)\ |\ Y\in \Sec(\D),\quad g(X,Y)=0, \quad k=0,1,2,\hdots>$$
is of constant rank $r$ along $q(t)$ and it does not contain $X(q(t))$ for any $t\in[0,T]$. 
	\item There exists a smooth  distribution $\B$ on $Q$, such that 
$$[X,\B]_{q(t)}\subset\B_{q(t)}\ ,\quad X(q(t))\notin \B_{q(t)}\quad\text{and}\quad \D^\perp_{q(t)}\subset\B_{q(t)}\quad \text{for any $t\in[0,T]$.}$$
\end{itemize}
\end{corollary}

\begin{proof}[Proof of Theorem \ref{thm:normal}.]
Assume first that $(\wt q(t),\wh u(t))$ is a normal SR extremal. Let $\wt\HH_t\subset\T_{\wt q(t)}\wt Q$ be the related curve of separating hyperplanes given by the PMP. Note that, since $\wt\HH_t$ for each $t$ is a hyperplane transversal to the line $\RR_{\wt q(t)}\subset \T_{\wt q(t)} \wt Q$, it must be of the form
$$\wt \HH_t=\{Y+\alpha_t(Y)\pa_{q_0}\ |\ Y\in\T_{q(t)}Q\}\ ,$$
where $\alpha_t:\T_{q(t)}Q\ra \R$ is a linear map.  Using the results of Lemma \ref{prop:dim_H} we know that $\alpha_t\big|_{\D_{q(t)}}=f_{\wh u(t)}(q(t))\Big\lrcorner g$, i.e., $\alpha_t(f_{\wh u(t)}(q(t)))=1$ and $\D^\perp_{q(t)}\subset \ker\alpha_t$. In particular, $f_{u(t)}(q(t))$ is transversal to $\ker\alpha_t\supset D^\perp_{q(t)}$. 

Now, since $\bm f_{\wh u(t)}=f_{\wh u(t)}+\frac 12\pa_{q_0}$ (here we use the normalization of $f_{\wh u(t)}$), it is clear that 
$$\Flow{t\tau}(q,q_0)=(\flow{t\tau}(q),q_0+\frac 12(t-\tau))\ .$$
It follows that $\T \Flow{t\tau}\left[Y+\alpha_{\tau}(Y)\pa_{q_0}\right]=\T \flow{t\tau}(Y)+\alpha_{\tau}(Y)\pa_{q_0}$, for every $t,\tau\in[0,T]$ and $Y\in \T_{q(t)}Q$. Since $\T \Flow{t\tau}\left(\wt\HH_{\tau}\right)=\wt \HH_t$, the above vector must be of the form $X+\alpha_t(X)\pa_{q_0}$, where $X=\T \flow{t\tau}(Y)$. That is, $\alpha_t(\T \flow{t\tau}(X))=\alpha_\tau(X)$. In particular, $t\mapsto \alpha_t$ is continuous and, moreover, $\T \flow{t\tau}\left(\ker \alpha_\tau\right)=\ker \alpha_t$ for every $t,\tau\in[0,T]$. We conclude that $\ker\alpha_t$ is a distribution along $q(t)$ which is $\flow{t\tau}$-invariant, contains $D^\perp_{q(t)}$ and is transversal to $f_{\wh u(t)}(q(t))$. Clearly, $\flow\bullet(D^\perp)_{q(t)}\subset\ker\alpha_t$ and thus it is also transversal to $f_{\wh u(t)}(q(t))$. 

To prove that $t\mapsto f_{\wh u(t)}(q(t))$ is ACB, observe first that $D^\perp_{q(t)}=\ker\alpha_t\cap D_{q(t)}$ admits locally a $g$-orthonormal basis of ACB sections. Indeed, $\ker\alpha_t$ is \sexy\ since it is $\flow{t\tau}$-invariant (cf. Proposition~\ref{prop:sexy_examples}). Let now $\{X_1,\hdots,X_{n-1}\}$ be a local basis of ACB sections of $\ker\alpha_t$ along $q(t)$. Choose a minimal subset of this basis, say $\{X_1,\hdots,X_s\}$, such that $\<X_1,\hdots,X_s>_{q(t)}\oplus \D^\perp_{q(t)}=\ker\alpha_t$ for every $t$ in a relatively compact neighborhood of a given point $t_0\in[0,T]$. Extend locally the SR metric $g$ to a metric $\wtilde g$ on $\ker \alpha_t$ by taking $\wtilde g\big|_{\D^\perp_{q(t)}}=g\big|_{\D^\perp_{q(t)}}$ and by setting vectors $X_1,\hdots,X_s$ to be $\wtilde g$-orthonormal and $\wtilde g$-orthogonal to $\D^\perp_{q(t)}$. Clearly, this new metric is ACB in the considered neighborhood of $t_0$. Now we can apply Lemma \ref{lem:gram_schmidt} to the ACB basis $\{X_1,\hdots,X_{n-1}\}$ and obtain an ACB $\wtilde g$-orthonormal basis $\{X_1,\hdots,X_s,Y_{s+1},\hdots,Y_{n-1}\}$ of $\ker\alpha_t$. Clearly, by the construction of the Gram-Schmidt algorithm, $\{Y_{s+1},\hdots,Y_{n-1}\}$ is a $\wtilde g$-, and thus also a $g$-orthonormal basis of $\D^\perp_{q(t)}$ (the relative compactness of the neighborhood is used to assure that the $\wtilde g$-lengths of sections $X_i$ are separated from zero).

Now let us choose any ACB section $Y_n$ of $\D_{q(t)}$ which is transversal to $\D^\perp_{q(t)}$. Again using  Lemma~\ref{lem:gram_schmidt} we modify the ACB local basis $\{Y_{s+1},\hdots,Y_{n-1},Y_n\}$ of $\D_{q(t)}$ to a $g$-orthonormal ACB local basis $\{Y_{s+1},\hdots,Y_{n-1},\wtilde Y_n\}$. Obviously, $\wtilde Y_n(q(t))$ is a $g$-normalized vector $g$-orthogonal to $\D^\perp_{q(t)}=\<Y_{s+1},\hdots,Y_{n-1}>$, thus $\wtilde Y_n(q(t))=\pm f_{u(t)}(q(t))$. Now $\alpha_t(\wtilde Y_n(q(t)))=\pm\alpha_t(f_{u(t)}(q(t)))=\pm 1$. And since both $\alpha_t$ and $\wtilde Y_n(q(t))$ are continuous with respect to $t$ the sign $\pm$ must be constant along $[0,T]$. We conclude that $t\mapsto f_{u(t)}(q(t))$ is ACB alike $t\mapsto\wtilde Y_n(q(t))$.  

\medskip
Conversely, assume that \eqref{cond:B_normal} holds.
{\new 
The condition that the velocity curve $t\mapsto f_{\wh u(t)}(q(t))$ is of class ACB, together with the normalization condition, imply that the velocities are preserved by the the flow $\flow{t \tau}$ up to $\flow\bullet(\D^\perp)_{q(t)}$-terms, i.e.,
\begin{equation}
\label{eqn:f_preserved}
\T \flow{t\tau}\left[f_{\wh u(\tau)}(q(\tau))\right]=f_{\wh u(t)}(q(t))\mod \flow\bullet(\D^\perp)_{q(t)}\qquad\text{for every $t,\tau\in[0,T]$}.
\end{equation}
Indeed, denote $\dot q(t)=f_{\wh u(t)}(q(t))=\sum_{i=1}^du^i(t)f_i(q(t))$, where $\{f_1,f_2,\hdots,f_d\}$ is a local basis of smooth sections of $\D$ and where controls $u^i(t)$'s are, by assumption, ACB with respect to $t$. In particular, the derivatives $\dot u^i(t)$ are a.e. well-defined, bounded, and measurable with respect to $t$. We can differentiate $\dot q(t)$ to get 
$$\ddot q(t)\overset{a.e}=\sum_{i=1}^du^i(t)\frac{\pa f_i(q)}{\pa q}\Big|_{q(t)}\dot q(t)+\sum_{i=1}^d\dot u^i(t)f_i(q(t))=\frac{\pa f_{\wh u(t)}(q)}{\pa q}\Big|_{q(t)}\dot q(t)+\sum_{i=1}^d\dot u^i(t)f_i(q(t))\,,$$
i.e., the velocity $\dot q(t)$ is a solution of an non-autonomous linear ODE with the right-hand side depending measurably on time. Note that the homogeneous part of this equation is precisely equation \eqref{eqn:tangent_ODE} describing the tangent map of the flow of $f_{\wh u(t)}(q)$ along $q(t)$. By the general theory of linear ODEs, for every $t,\tau\in[0,T]$  we have
$$\dot q(t)=\T\flow{t \tau}[\dot q(\tau)]+\int_\tau^t\T\flow{ts}\left[\sum_{i=1}^d\dot u^i(s)f_i(q(s))\right]\dd s \,. $$ 
By the normalization condition vector $\sum_{i=1}^d\dot u^i(s)f_i(q(s))$ is $g$-perpendicular to $\dot q(s)=\sum_{i=1}^d u^i(s)f_i(q(s))$, and hence the integral part of the above equation belongs to $\flow\bullet(D^\perp)_{q(t)}$. This proves \eqref{eqn:f_preserved}. 
}

 The crucial step now is to build, along the projected trajectory $q(t)\in Q$, a splitting $\T_{q(t)} Q=\mathcal{B}_{q(t)}\oplus \<f_{\wh u(t)}(q(t))>$, where $\B_{q(t)}$ is a co-rank one distribution along $q(t)$, which is $\flow{t\tau}$-invariant along $q(t)$ and contains $\D^\perp_{q(t)}$. Such a $\mathcal{B}_{q(t)}$ can be constructed by adding, if necessary, to $\flow\bullet(\D^\perp)_{q(t)}$ several vector fields of the form $F_{t0}(X_i)$, where $X_i \in T_{q(0)} Q$ together with $F_{\hat u (0)} (q(0))$ are independent. Clearly, in this way we can build $\mathcal{B}_{q(t)}$ which is $\flow{t\tau}$-invariant along $q(t)$, of co-rank one and contains $\D^\perp_{q(t)}$. {\new The fact that $\mathcal{B}_{q(t)}$ and $f_{\wh u(t)}(q(t))$ are linearly independent follows directly from condition \eqref{cond:B_normal}. }
Now we can construct the curve of separating hyperplanes $\bm\HH_t$ at $\wt q(t)$ by the formula
$$\bm \HH_t:=(\mathcal B_{q(t)}\oplus 0\cdot \pa_{q_0})\oplus\<f_{\wh u(t)}(q(t))+1\cdot\pa_{q_0}>\ .$$
By construction it is clear that $\bm \HH_t$ is a hyperplane in $\T_{\bm q(t)}\bm Q$ which contains the tangent space to the paraboloid \eqref{eqn:T_cone} and does not contain the line $\bm \RR_{\bm q(t)}$. From these properties we conclude that $\bm \HH_t$ separates strictly the ray $\bm \RR^-_{\bm q(t)}$ from the elements of $\bm \K_t$ of the form $\bm f_{v}(q(t))-\bm f_{\wh u(t)}(q(t))$. {\new Observe that thanks to the normalization of $f_{\wh u(t)}$ (cf. the first part of this proof) and condition \eqref{eqn:f_preserved} we have 
$$\T\Flow{t\tau}\left[\<f_{\wh u(\tau)}(q(\tau))+1\cdot\pa_{q_0}>\right]=\<f_{\wh u(t)}(q(t))+1\cdot\pa_{q_0}>\mod (\mathcal B_{q(t)}\oplus 0\cdot \pa_{q_0})\,,$$
and thus $\bm \HH_t$ is $\T\Flow{t\tau}$-invariant. We conclude that the hyperplane $\bm \HH_t$} also separates strictly the ray $\bm \RR^-_{\bm q(t)}=\T \Flow{t\tau}(\bm \RR^-_{q(\tau)})$ from the elements of $\bm \K_t$ of the form $\T \Flow{t\tau}[\bm f_{v}(\bm q(\tau))-\bm f_{\wh u(t)} (\bm q(\tau))]$. Consequently, using the fact that $\bm \HH_t$ and $\bm \RR^-_{\bm q(t)}$ are convex, we can use $\bm \HH_t$ to separate strictly $\bm \RR^-_{\bm q(t)}$ from any finite convex combination of the above-mentioned elements of $\bm \K_t$. Since $\bm \K_t$ is by definition the closure of the set of such finite convex combinations, $\bm \HH_t$ is indeed the separating  hyperplane described by the PMP. 
\end{proof} 

\paragraph{A remark on smoothness of normal SR geodesics.}
As was proved above normal SR extremals are $C^1$-smooth (and even more: their derivatives are ACB maps). 
It is worth discussing geometric reasons for this regularity in a less technical manner than in the proof of Theorem \ref{thm:normal}. Let $\wt q(t)$ be such an extremal and let $\wt \HH_t$ be the corresponding curve of supporting hyperplanes. As we know from Lemma \ref{prop:dim_H}
$$\D^\perp_{q(t)}\oplus 0\cdot\pa_{q_0}\subset\wt\HH_t\quad\text{and}\quad f_{u(t)}(q(t))+\pa_{q_0}\in\wt \HH_t\quad \text{for every $t\in[0,T]$}.$$ 
\begin{figure}[ht]%
\begin{center}
\def\svgwidth{0.9 \columnwidth}   
\input{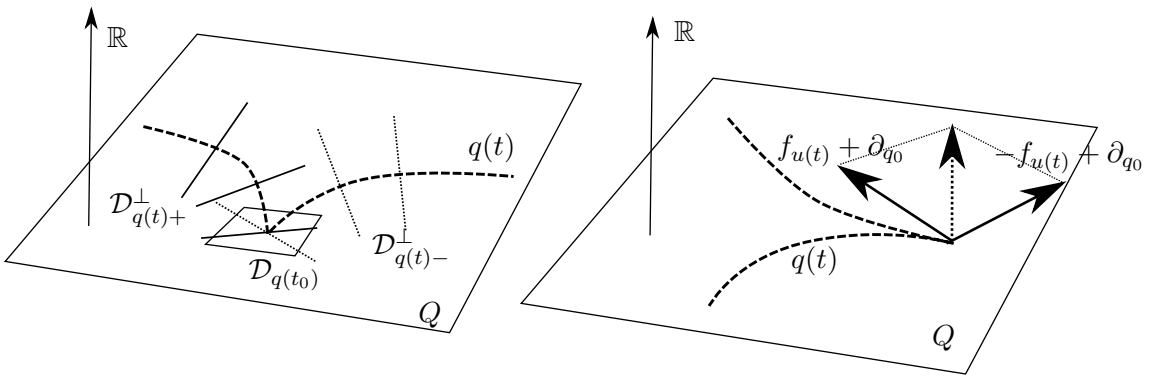}
\caption{The existence of singularities of corner- or cusp-type implies abnormality or the lack of optimality.}
\label{fig:singular}%
\end{center}
\end{figure}
These two facts are enough to exclude, at least in a heuristic way, the existence of singularities of corner-type and of cusp-type along $q(t)$. Indeed, since $\wt\HH_t$ is $\Flow{t\tau}$-invariant it must be continuous. Note that by the continuity of $\wt \HH_t$, the limit subspaces $\D^\perp_{q(t_0)\pm}\oplus 0\cdot\pa_{q_0}$ coming from both sides of a given point $t_0\in[0,T]$ must belong to $\wt\HH_{t_0}$. 
Now if $\wt q(t)$ had a corner-type singularity at $t_0$, these limit subspaces would be different and thus they would span together the whole space $\D_{q(t_0)}\oplus 0\cdot\pa_{q_0}$ (cf. Figure \ref{fig:singular}). In particular, $f_{u(t_0)}(q(t_0))+0\cdot\pa_{q_0}\in \D_{q(t_0)}\oplus 0\cdot\pa_{q_0}$ would belong to $\wt\HH_{t_0}$. Yet, since $f_{u(t_0)}(q(t_0))+\pa_{q_0}\in\wt \HH_{t_0}$, this would mean that also the difference of the latter vectors, $0+\pa_{q_0}$, lies in $\wt\HH_{t_0}$, which is impossible since $\wt q(t)$ is normal. 

In a similar way one deals with a cusp-type singularity. At a cusp we would have limit vectors $\pm f_{u(t_0)}(q(t_0))+\pa_{q_0}$ in $\wt\HH_{t_0}$ (see Figure \ref{fig:singular}). Now $0+2\pa_{q_0}$, the sum of these two vectors, would belong to $\wt\HH_{t_0}$ which contradicts the normality of the extremal.   
Roughly speaking, the existence of singularities of corner-type or cusp-type implies $\pa_{q_0}\in\HH_{t_0}$, i.e., either a trajectory is not an extremal or it is abnormal.

\paragraph{Examples.}

\begin{example}(Geodesic equation revisited)
\label{ex:geod1}
Theorem \ref{thm:normal} provides an alternative way to derive the geodesic equation in the Riemannian case (i.e., when $\D=\T Q$). Let $(\wt q(t),\wh u(t))$ be a trajectory of the SR control system (we shall assume that $f_{\wh u(t)}$ is normalized). Since $\D=\T Q$, by the assertion of Theorem \ref{thm:abnormal}, in the Riemannian case there are no abnormal extremals.

Since $\D^\perp_{q(t)}=\{f_{\wh u(t)}(q(t))\}^\perp$ is of co-rank one, the only distribution of higher rank along $q(t)$  containing $\D^\perp_{q(t)}$ is $\T_{q(t)} Q$ which contains also $f_{\wh u(t)}(q(t))$. Now, by Theorem \ref{thm:normal},  $(\wt q(t),\wh u(t))$ is a normal extremal if and only if $\flow\bullet(\D^\perp)_{q(t)}=\D^\perp_{q(t)}$, i.e., if
$$\T \flow{t\tau}(\D^\perp_{q(\tau)})=\D^\perp_{q(t)}\ ,$$
for every $t,\tau\in[0,T]$. By the results of Theorem \ref{thm:flow_bracket} this is equivalent to
$$[f_{\wh u(t)},\D^\perp]_{q(t)}=\D^\perp_{q(t)}\ ,$$
i.e., $g([f_{\wh u(t)},Y],f_{\wh u(t)})_{q(t)}=0$ whenever $g(Y,f_{\wh u(t)})_{q(t)}=0$. Now for such a $Y$, after introducing a metric-compatible connection as in Example \ref{ex:geod}, we have 
\begin{align*}
0=&g([f_{\wh u(t)},Y],f_{\wh u(t)})=g(\nabla_{f_{\wh u(t)}}Y-\nabla_Yf_{\wh u(t)}-T_\nabla(f_{\wh u(t)}, Y),f_{\wh u(t)})=\\
&g(\nabla_{f_{\wh u(t)}}Y,f_{\wh u(t)})-g(\nabla_Yf_{\wh u(t)},f_{\wh u(t)})-g(T_\nabla(f_{\wh u(t)}, Y),f_{\wh u(t)})=\\
&f_{\wh u(t)} g(Y,f_{\wh u(t)})-g(Y,\nabla_{f_{\wh u(t)}}f_{\wh u(t)})-\frac 12 Yg(f_{\wh u(t)},f_{\wh u(t)})-g(T_\nabla(f_{\wh u(t)}, Y),f_{\wh u(t)})\ .
\end{align*}
Using the fact that $g(Y,f_{\wh u(t)})\equiv 0$ and that $g(f_{\wh u(t)}, f_{\wh u(t)})\equiv 1$ we get
$$g(Y,\nabla_{f_{\wh u(t)}}f_{\wh u(t)})+g(T_\nabla(f_{\wh u(t)}, Y),f_{\wh u(t)})=0$$
in agreement with the results of Example \ref{ex:geod}. 
\end{example}

\begin{example}[Heisenberg system]
\label{ex:heisenberg}
Consider a SR system on $\R^3\ni(x,y,z)$ constituted by a 2-distribution 
$$\D_{(x,y,z)}=\<Y:=\pa_x-y\pa_z,Z:=\pa_y+x\pa_z>$$
and a SR metric such that the fields $Y$ and $Z$ form an orthonormal basis. Such a system is usually called the \emph{Heisenberg system}. It is easy to check that the system in question is strongly bracket generating (cf. Example \ref{ex:sbg}) and as such does not admit any abnormal SR extremal. Our goal will  thus be to determine the normal SR extremals using the results of Theorem \ref{thm:normal}. 

Take now any normalized $\D$-valued vector field $X=X_t:=\phi(t) Y+\psi(t) Z$, where $\phi^2+\psi^2=1$. We have $\D^\perp=\<X'={X_t}':=\psi(t) Y-\phi(t) Z>$ and, by the results of Theorem \ref{thm:normal}, the integral curve $q(t)$ of $X$ is a SR normal extremal if and only if $\flow{\bullet}(\D^\perp)_{q(t)}$ does not contain $X$ at any point $q(t)$. Clearly  distribution $\flow{\bullet}(\D^\perp)_{q(t)}$, being $\ad_X$-invariant, contains the fields $X'$, $[X,X']$, $[X,[X,X']]$, etc. Skipping some simple calculations one can show that
$$[X,X']=-2\pa_z+AY+BZ\ ,$$
where 
\begin{equation}
\label{eqn:AB}
\begin{split}
A&=\phi Y(\psi)-\psi Y(\phi)+\frac 12 Z(\phi^2+\psi^2)=\phi Y(\psi)-\psi Y(\phi) \\
B&= \phi Z(\psi)-\psi Z(\phi)-\frac 12 Y(\phi^2+\psi^2)=\phi Z(\psi)-\psi Z(\phi)\ . 
\end{split}
\end{equation}
Let us now present vector field $[X,X']$ as (note that $\{X,X'\}$ is a basis of sections of $\D$)
$$[X,X']=-2\pa_z+\alpha X+\beta X'\ .$$
Then 
$$[X,[X,X']]=X(\alpha)X+\beta[X,X']+X(\beta) X'\ .$$
Now clearly $\<X',[X,X'],[X,[X,X']]>$ would contain $X$ if and only if $X(\alpha)\neq 0$. Thus a necessary condition for an integral curve $q(t)$ of $X$ to be a normal SR extremal is that $\alpha=const$ along $q(t)$. Note that if $X(\alpha)=const$, then the integral curves of $X$ will indeed be normal SR extremals, as then $[X,[X,X']]=\beta[X,X']+X(\beta) X'$ and, consequently, $\flow\bullet(\D^\perp)_{q(t)}$ will be equal to the 2-dimensional distribution $\<X',[X,X']>$ which does not contain $X$ (cf. Corollary \ref{cor:normal}).

By comparing the coefficients of $[X,X']$ expressed in terms of the bases $\{Y,Z\}$ and $\{X,X'\}$ we get
$$A Y+B Z=\alpha X+\beta X'=\alpha(\phi Y+\psi Z)+\beta (\psi Y-\phi Z)=(\alpha \phi+\beta \psi)Y+(\alpha \psi-\beta \phi)Z\ .$$
Thus, by \eqref{eqn:AB},
\begin{align*}
\phi Y(\psi)-\psi Y(\phi)-\alpha \phi&=\beta \psi \\
\phi Z(\psi)-\psi Z(\phi)-\alpha \psi&=-\beta \phi\ .
\end{align*} 
Consequently, 
$$\phi^2 Y(\psi)-\phi\psi Y(\phi)-\alpha \phi^2=\beta \phi\psi=-\phi\psi Z(\psi)+\psi^2 Z(\phi)+\alpha \psi^2,$$
which, after substituting $\phi Y+\psi Z$ by $X$, leads to
$$X(\psi/\phi)=\frac {\phi X(\psi)-\psi X(\phi)}{\phi^2}=\alpha(1+(\psi/\phi)^2)\ ,$$
i.e., the quotient $\psi/\phi$ satisfies the equation $X(x)=\alpha(1+x^2)$, where $\alpha$ is a constant. For $\alpha=0$ we get $x=const$ (i.e., $\phi$ and $\psi$ are constant along $q(t)$), and for $\alpha\neq 0$ we get $x=\arctan(\alpha t+\gamma)$ (i.e., $\phi=\cos(\alpha t+\gamma)$ and $\psi=\sin(\alpha t+\gamma)$). This corresponds to the two well-known families of normal SR extremals of the Heisenberg system (see Sec. 2 of \cite{Liu_Sussmann_1995}), whose projections to the $(x,y)$-plane are straight lines and circles, respectively.
\end{example}

\appendix
\section{Technical results}
\label{sec:appendix}

Below we present technical results and their proofs used in the course of our considerations in Section \ref{sec:technical}.  

\paragraph{Measurable maps.}
We shall start by recalling some basic definitions and results from function theory. 

A map $f:\R\supset V\ra \R^n$ defined on an open subset $V\subset \R$ is called \emph{measurable} if the inverse image of every open set in $\R^n$ is Lebesgue-measurable in $V$. We call $f$ \emph{bounded} if the closure of the image $f(V)$ is a compact set, and \emph{locally bounded} if the closure of the image of every compact set is compact. A point $t\in V$ is called a \emph{regular point} of $f:V\ra \R^n$, if for every open neighborhood $O\in f(t)$, we have
$$\lim_{\operatorname{diam}(V')\to 0}\frac{\mu(f^{-1}(O)\cap V')}{\mu( V')}=1\ .$$
Here the limit is taken over open neighborhoods $V'\ni t$ and $\mu(\cdot)$ denotes the Lebesgue measure on $V$. By Lebesgue theorem, the set of regular points of a bounded and measurable map $f:V\ra \R^n$ is of full measure in $V$. 

A map $x:\R\supset[t_0,t_1]\ra \R^n$ is called \emph{absolutely continuous} (AC, in short) if 
it can be presented in a form of an integral 
$$x(t)=x(t_0)+\int_{t_0}^t v(s)\dd s\ ,$$
for some integrable map $v(\cdot)$. Clearly, an AC map is differentiable at all  regular points $t$ of $v$ (and thus, by Lebesgue theorem, a.e.) and the derivative of $x(t)$ at such a point is simply $v(t)$. We will be particularly interested in AC maps $x(t)$ such that the derivative $v(t)$ is locally bounded. In such a case we shall speak about \emph{AC maps with bounded derivative} (ACB, in short).

\paragraph{Measurable ODE's.}
While speaking about ODE's in the measurable setting we will need to take care of some technical properties of certain functions. In order to simplify the discussion let us introduce the following
\begin{definition}
\label{def:nice_map} 
A map $F:\R^n\times\R\ra\R^m$ will be called \emph{\nice} if the assignment $(x,t)\mapsto F(x,t)$ is 
\begin{align}
  \label{cond:A}
	\text{locally bounded, differentiable with respect to $x$, and measurable with respect to $t$}
\intertext{and if the derivative $(x,t)\mapsto\frac{\pa F}{\pa x}(x,t)$ is}
	\label{cond:B}
	\text{locally bounded, continuous with respect to $x$, and measurable with respect to $t$.}
\end{align}
The notion of a \nice\ map can be naturally extended to the setting of smooth manifolds, namely we shall call a map $F:M\times\R\ra N$ \emph{\nice} if it is \nice\ in a (and thus in any) local smooth coordinate chart on $M$ and $N$. Indeed, it is easy to see that this property does not depend on the particular choice of a chart (cf. the notion of a Caratheodory section in \cite{Jafarpour_Lewis_2014}). 
\end{definition}

Consider now a map $G:\R^n\times\R\ra\R^n$ and the associated non-autonomous ODE in $\R^n$
\begin{equation}
\label{eqn:ODE}
\dot x(t)=G(x(t),t)\ .
\end{equation}
By a (\emph{Caratheodory}) \emph{solution} of \eqref{eqn:ODE} on $[t_0,t_1]$ with the initial condition $x_0$ at $t_0$ we shall understand an AC map $[t_0,t_1]\ni t\mapsto x(t)\in \R^n$ which satisfies \eqref{eqn:ODE} a.e. (recall that an AC map is differentiable a.e.), such that $x(t_0)=x_0$. Note that speaking about Caratheodory solutions makes sense also if the map $G$ is defined only a.e..  

The following fact is a straightforward generalization, to the measurable context, of the standard result about the existence and uniqueness of the solutions of ODE's.

\begin{theorem}\label{thm:solutions_ODE} Assume that the map $(x,t)\mapsto G(x,t)$ is \nice.
Then, for each choice of $(t_0,x_0)\in\R\times\R^n$ there exists, , in a neighborhood of $t_0$, a unique (Caratheodory) ACB solution $t\mapsto x(t;t_0,x_0)$ of equation \eqref{eqn:ODE} satisfying $x(t_0;t_0,x_0)=x_0$.

Moreover, $x(t;t_0,x_0)$ is differentiable with respect to $x_0$ and the derivative $\frac{\pa x}{\pa x_0}(t;t_0,x_0)$ is continuous with respect to $x_0$ and ACB with respect to $t$. In fact, the derivative $\frac{\pa x}{\pa x_0}(t;t_0,x_0)$ is the unique (Caratheodory) solution of the following linear time-dependent ODE, called the variational equation,
\begin{equation}
\label{eqn:tangent_ODE}
\dot V(t,x_0)=\frac{\pa G}{\pa x}(x(t;t_0,x_0),t)V(t,x_0)
\end{equation}
for a curve of linear maps $V(t,x_0):\T_{x_0} \R^n\ra \T_{x(t;t_0,x_0)}\R^n$ 
with the initial condition $V(t_0,x_0)=\id_{\T_{x_0} \R^n}$.
\end{theorem}
The proof is given in \cite{Bressan_Piccoli_2004} (Theorem 3.3.2). Also Sec. 3 of \cite{Grabowski_Jozwikowski_2011} may be useful. Note that equation \eqref{eqn:tangent_ODE} can be obtained by differentiating  the equation $\dot x(t;t_0,x_0)=G(x(t;t_0,x_0),t)$ with respect to $x_0$ and  substituting $V(t,x_0)$ for $\frac{\pa x}{\pa x_0}(t;t_0,x_0)$.

\paragraph{Proof of Theorem \ref{thm:flow_bracket}.}\
In this paragraph we will provide a rigorous proof of Theorem \ref{thm:flow_bracket}. We shall begin with  the following lemma which characterizes the tangent map $\T \A{t\tau}$ of the TD flow of a TDVF $X_t$ in terms of the Lie bracket $[X_t,\cdot]$. Informally speaking, transporting a given vector $Z_0$ via the map $\T \A{tt_0}$ along an integral curve of $X_t$ turns out to be the same as solving the equation $[X_t,\cdot]=0$.

\begin{lemma}\label{lem:bracket_single}
Let $X_t$ be a \nice\ TDVF on a manifold $M$,  $x(t)=x(t;t_0,x_0)$ (with $t\in[t_0,t_1]$) its integral curve, and $\A{tt_0}$ its TD flow. Let $Z_0\in \T_{x_0}M$ be a tangent vector at $x_0$ and denote by $Z(x(t))$ a vector field along $x(t)$ obtained from $Z_0$ by the action of the TD flow $\A{tt_0}$, i.e., $Z(x(t)):=\T \A{tt_0}(Z_0)$. Then the assignment $t\mapsto Z(x(t))$ is ACB and, moreover,
\begin{equation}
\label{eqn:bracket_zero}
[X_t,Z]_{x(t)}=0\quad\text{for a.e. $t\in[t_0,t_1]$.}
\end{equation}

Conversely, if $Z$ is a vector field along $x(t)$ such that the assignment $t\mapsto Z(x(t))$ is ACB and  that equation \eqref{eqn:bracket_zero} holds, then $Z(x(t))=\T \A{tt_0}\left(Z_0\right)$, where $Z_0=Z(x(t_0))$.
\end{lemma} 

\begin{proof}
Consider first the vector field $Z(x(t)):=\T \A{tt_0}(Z_0)$ along $x(t)$. The fact that $t\mapsto Z(x(t))$ is ACB follows directly from the second part of the assertion of Theorem \ref{thm:solutions_ODE}. 

Let $s\mapsto z_0(s)$ be a curve in $M$ representing $Z_0$, i.e., $z_0(0)=x_0$ and $\frac{\pa}{\pa s}\big|_{s=0}z_0(s)=Z_0$. It is clear that for each $t\in\R$ the vector $Z(x(t))=\T \A{tt_0}(Z_0)$ is represented by the curve $s\mapsto x(t;t_0,z_0(s))=\A{tt_0}(z_0(s))$.
Now from \eqref{eqn:lie_bracket} we have
\begin{align*}
[X_t,Z]_{x(t)}\overset{a.e.}=&\frac{\pa}{\pa t}Z(x(t))-\frac{\pa}{\pa s}\Big|_{s=0}X_t(x(t;t_0,z_0(s)))=\\
&\frac{\pa}{\pa t}\frac{\pa}{\pa s}\Big|_{s=0}x(t;t_0,z_0(s))-\frac{\pa}{\pa s}\Big|_{s=0}X_t(x(t;t_0,z_0(s)))=\\
&\frac{\pa}{\pa t}\left(\frac{\pa x}{\pa x_0}x(t;t_0,x_0)\frac{\pa}{\pa s}\Big|_{s=0}z_0(s)\right)-\frac{\pa}{\pa s}\Big|_{s=0}X_t(x(t;t_0,z_0(s)))=\\
&\frac{\pa}{\pa t}\left(\frac{\pa x}{\pa x_0}x(t;t_0,x_0)Z_0\right)-\frac{\pa}{\pa s}\Big|_{s=0}X_t(x(t;t_0,z_0(s)))=\\
&\frac{\pa}{\pa t}\left(\frac{\pa x}{\pa x_0}x(t;t_0,x_0)\right)Z_0-\frac{\pa}{\pa s}\Big|_{s=0}X_t(x(t;t_0,z_0(s)))\ .
\end{align*}
Passing to local coordinates in which $X_t(x)$ writes as $G(x,t)$ and using the fact that $\frac{\pa x}{\pa x_0}(t;t_0, x_0)$ satisfies \eqref{eqn:tangent_ODE}, we easily get
\begin{align*}
[X_t,Z]_{x(t)}\overset{a.e.}=&\frac{\pa G}{\pa x}(x(t;t_0,x_0),t)\frac{\pa x}{\pa x_0}(t;t_0,x_0)Z_0-\frac{\pa}{\pa s}\Big|_{s=0}G(x(t;t_0,z_0(s)),t)=\\
&\frac{\pa G}{\pa x}(x(t;t_0,x_0),t)\frac{\pa x}{\pa x_0}(t;t_0,x_0)Z_0-\frac{\pa G}{\pa x}(x(t;t_0,z_0(0)),t)\frac{\pa x}{\pa x_0}(t;t_0,z_0(0))\frac{\pa}{\pa s}\Big|_{0}z_0(s)=0\ ,
\end{align*} 
as $z_0(0)=x_0$ and $\frac{\pa}{\pa s}\big|_0z_0(s)=Z_0$.
\medskip

To prove the opposite implication let now $t\mapsto Z(x(t))$ be an ACB vector field along $x(t)$ which commutes with $X_t$. Let us choose a family of curves $s\mapsto z(t,s)$ representing vectors $ Z(x(t))$ for each $t$, that is $z(t,0)=x(t)$ and $\frac{\pa}{\pa s}\big|_{s=0}z(t,s)= Z(x(t))$. Since $[X_t, Z]_{x(t)}=0$, we have by \eqref{eqn:lie_bracket} 
$$\frac{\pa}{\pa t} Z(x(t))\overset{a.e.}=\frac{\pa}{\pa s}\Big|_{s=0}X_t(z(t,s))\ .$$
After introducing local coordinates as above we have
\begin{align*}
\frac{\pa}{\pa t}Z(x(t))\overset{a.e.}=\frac{\pa }{\pa s}\Big|_{s=0}G(z(t,s),t)=\frac{\pa G}{\pa x}(z(t,0),t)\frac{\pa}{\pa s}\Big|_{s=0}z(t,s)=\frac{\pa G}{\pa x}(x(t),t)Z(x(t))\ .
\end{align*}
As we see $t\mapsto Z(x(t))$ satisfies the linear ODE
\begin{equation}
\label{eqn:linear_ODE}
\frac{\pa}{\pa t}W(t)\overset{a.e.}=\frac{\pa G}{\pa x}(x(t),t)W(t)\ .
\end{equation}

Since, by the first part of this proof, for the vector field $\wtilde Z(x(t)):=\T
\A{tt_0}\left[Z(x(t_0))\right]$ we also have $[X_t,\wtilde Z]_{x(t)}=0$ a.e. along $x(t)$, we conclude that $t\mapsto \wtilde Z(x(t))$ is also subject to a linear ODE of the form \eqref{eqn:linear_ODE}. Thus the difference $Z(x(t))-\wtilde Z(x(t))$ is a Caratheodory solution of the linear ODE \eqref{eqn:linear_ODE} with the initial value $Z(x(t_0)-\wtilde Z(x(t_0))=0$. Using the uniqueness of the solution (cf. Theorem \ref{thm:solutions_ODE}) we conclude that $Z(x(t))-\wtilde Z(x(t))\equiv 0$. 
\end{proof}

Now we are finally ready to prove Theorem \ref{thm:flow_bracket}. 

\begin{proof}[Proof of Theorem \ref{thm:flow_bracket}.]
Assume first that condition \eqref{cond:a} of Theorem \ref{thm:flow_bracket} holds. Choose a basis $\{Z_{10},\hdots,Z_{k0}\}$ of $\B_{x(t_0)}$, where $k$ is the rank of $\B$, and for $i=1\hdots,k$ denote $Z_{i}(x(t)):=\T \A{tt_0}(Z_{i0})$. By the results of Lemma \ref{lem:bracket_single}, the fields $Z_i$ are ACB along $x(t)$ and satisfy $[X_t,Z_i]_{x(t)}\equiv 0$ a.e. along $x(t)$.  Thanks to condition \eqref{cond:a} and the fact that $\A{tt_0}$ is a local diffeomorphism,  $Z_i$'s span $\B$. 

Let now $Z\in\Sec_{ACB}(\B)$ be any ACB section of $\B$. We want to present it as a linear combination of fields $Z_i$ with ACB coefficients, i.e., $Z=\sum_i \phi^i Z_i$, where $\phi^i$ are ACB functions along $x(t)$. To prove that such a presentation is possible first take vectors $W_{j0}\in \T _{x(t_0)} M$ with $j=0,\hdots,s$ such that $Z_{i0}$'s together with $W_{j0}$'s form a basis of $\T_{x(t_0)}M$. Clearly, the fields $W_j(x(t)):=\T \A{tt_0}(W_{j0})$  together with $Z_i$'s  span $\T M$ along $x(t)$. Since by Lemma \ref{lem:bracket_single} these fields are ACB, given any local basis of smooth vector fields $\mathcal{U}:=\{U_1,\hdots, U_{k+s}\}$ on $M$, the transition matrix $T_{\mathcal{U}\to \mathcal{ZW}}$ from the basis $\mathcal{U}$ to the basis $\mathcal{ZW}:=\{Z_1,\hdots,Z_k,W_1,\hdots,W_s\}$ is a matrix of ACB functions. As $T_{\mathcal{U}\to \mathcal{ZW}}$ is non-degenerate, the inverse matrix $T_{\mathcal{ZW}\to \mathcal{U}}$ is also a matrix of ACB functions (here we use the fact that if $\phi$ is an ACB function separated from 0, then so is $\frac 1\phi$). Thus any vector field with ACB coefficients in basis $\mathcal{U}$ (in particular $Z$) will have ACB coefficient in basis $\mathcal{ZW}$. As the field $Z$ is $\B$-valued, all $W_j$'s coefficients of $Z$ vanish, i.e., $Z=\sum_i \phi^i Z_i$, where $\phi^i$ are ACB functions along $x(t)$ as intended. Now by the Leibniz rule\footnote{We leave the proof of the fact that the Lie bracket \eqref{eqn:lie_bracket} satisfies the Leibniz rule as an exercise.} we get
$$[X_t,Z]_{x(t)}=[X_t,\sum_i\phi^i Z_i]_{x(t)}=\sum_i\left(\phi^i[X_t,Z_i]_{x(t)}+X_t(\phi^i) Z_i\Big |_{x(t)}\right)\overset{a.e.}=\sum_iX_t(\phi^i)Z_i\Big |_{x(t)}\in \B_{x(t)}\ .$$
Thus \eqref{cond:a} implies \eqref{cond:b}. 
\medskip

Assume now that condition \eqref{cond:b} of Theorem \ref{thm:flow_bracket} holds. Let $\{\wtilde Z_1,\hdots,\wtilde Z_k\}$ be any basis of ACB sections of $\B$. The idea is to modify this basis to another basis of ACB sections $\{Z_1,\hdots,Z_k\}$, such that for every $i=1,\hdots, k$ we have $[Y_t,Z_i]_{x(t)}\equiv 0$ a.e. along $x(t)$. In the light of Lemma \ref{lem:bracket_single} this would imply that the new basis is respected by the flow $\A{tt_0}$ and, consequently, that \eqref{cond:a} holds.  

Due to \eqref{cond:b},  $[Y_t,\wtilde Z_i]_{x(t)}$ is a $\B$-valued locally bounded measurable vector field for each $i=1,\hdots,k$ and thus there exists a $k\times k$ matrix of locally bounded measurable functions\footnote{The existence of measurable functions $\phi_i^{\ j}$ can be justified in a similar manner to the existence of ACB functions $\phi_i$ above.} $\phi_i^{\ j}$ along $x(t)$ such that 
$$[X_t,\wtilde Z_i]_{x(t)}=\sum_j \phi_i^{\ j} \wtilde Z_j\Big |_{x(t)}\ .$$ 
Now the simple idea is to look for the desired basis $\{Z_1,\hdots,Z_k\}$ in the form $Z_i=\sum_j \psi_i^{\ j} \wtilde Z_j$, where $\psi_i^{\ j}$ is an invertible $k\times k$ matrix of function ACB along $x(t)$. Clearly for such $Z_i$'s we have
\begin{align*}
[X_t, Z_i]_{x(t)}=&[X_t,\sum_j \psi_i^{\ j} \wtilde Z_j]_{x(t)}=\sum_jX_t(\psi_i^{\ j})\wtilde Z_j\Big |_{x(t)}+\sum_j\psi_i^{\ j}[X_t,\wtilde Z_j]_{x(t)}\overset{a.e.}=
\\&\sum_jX_t(\psi_i^{\ j})\wtilde Z_j\Big |_{x(t)}+\sum_j\sum_s\psi_i^{\ j} \phi_j^{\ s}\wtilde Z_s\Big |_{x(t)}\ .
\end{align*} 
As we see $[X_t, Z_i]_{x(t)}=0$ a.e. along $x(t)$ if and only 
$$X_t(\psi_i^{\ j})\overset{a.e.}=-\sum_{s}\psi_i^{\ s}\phi_s^{\ j}\ ,$$
i.e., the matrix $\psi_i^{\ j}$ should be a solution of a linear ODE with locally bounded measurable coefficients. Due to the results of Theorem \ref{thm:solutions_ODE}, for a given initial condition, say, $\psi_i^{\ j}(x_0)=\delta_i^{\ j}$, this equation has a unique local ACB solution. As a consequence, we prove the local existence of the desired basis $\{Z_1,\hdots, Z_k\}$, which implies \eqref{cond:a}. 
\end{proof}

\paragraph{Proof of Lemma \ref{lem:along_x}.}

We shall end our considerations by providing the following

\begin{proof}[Proof of Lemma \ref{lem:along_x}.]
The idea of the proof is very simple. Consider another  $\D$-valued \nice\ TDVF $X'_t$ such that $X_t=X'_t$ along $x(t)$. We shall show that $\B$ is $X_t$-invariant if and only if it is $X'_t$-invariant along $x(t)$. The justification of this statement is just a matter of a calculation. Observe that since $X_t$ and $X'_t$ are both \nice\ and $\D$-valued, then so is their difference $X_t-X'_t$. Given any local basis of smooth vector fields $\{W_1,\hdots,W_s\}$  of $\D$ we may locally represent $X_t-X'_t$ as
$$X_t(x)-X'_t(x)=\sum_i\phi^i(t,x)W_i(x)\ ,$$
where $\phi^i(t,x)$ are \nice\ functions (in the sense of Definition \ref{def:nice_map}). Since $X'_t=X_t$ along $x(t)$ and $W_i$'s form a basis of $\D$, we have 
\begin{equation}
\label{eqn:f_0}
\phi^i(t,x(t))=0\ .
\end{equation}
Now for any section $Z\in\Sec_{ACB}(\B)$, using the same notation as in formula \eqref{eqn:lie_bracket}, we have
\begin{align*}
&[X'_t,Z]_{x(t)}-[X_t,Z]_{x(t)}\overset{\eqref{eqn:lie_bracket}}=\frac{\pa}{\pa s}\Big|_{s=0}X_t(z(t,s))-\frac{\pa}{\pa s}\Big|_{s=0}X'_t(z(t,s))=\\
&\frac{\pa}{\pa s}\Big|_{s=0}\left[X_t(z(t,s))-X'_t(z(t,s))\right]=\frac{\pa}{\pa s}\Big|_{s=0}\left[\sum_i \phi^i(t,z(t,s)) W_i(z(t,s))\right]=\\
&\sum_i\left[\frac{\pa}{\pa s}\Big|_{s=0}\phi^i(t,z(t,s)) W_i(x(t)))+ \phi^i(t,x(t)) \frac{\pa}{\pa s}\Big|_{s=0}W_i(z(t,s))\right]\overset{\eqref{eqn:f_0}}=\\
&\sum_i\frac{\pa}{\pa s}\Big|_{s=0}\phi^i(t,z(t,s)) W_i(x(t))\ .
\end{align*}
Clearly the above expression is $\D_{x(t)}\subset \B_{x(t)}$-valued. Thus along $x(t)$
$$[X_t,Z]_{x(t)}=[X'_t,Z]_{x(t)}\mod \B_{x(t)}\ ,$$
and hence, since $Z$ was an arbitrary ACB section of $\B$,
$$[X_t,\B]_{x(t)}=[X'_t,\B]_{x(t)}\mod\B_{x(t)}\ .$$
This ends the proof. 
\end{proof}

\paragraph{A technical result about \sexy\ distributions.}The following result will be needed in the course of Subsection \ref{ssec:normal} to prove that normal SR extremals are of class $C^1$. It states that the Gram-Schmidt orthogonalization algorithm works well on charming distributions.  
\begin{lemma}\label{lem:gram_schmidt}
Let $\B\subset\T M$ be a charming distribution along a curve $x:[t_0,t_1]\ra M$. Assume that $\B$ is equipped with a positively-defined scalar product $g:\B\times_{x(\cdot)}\B\ra \R$ such that that the assignment $t\mapsto g(x(t))$ is an ACB map. Let $\{X_1,\hdots,X_s\}$ be a family of $s$ linearly-independent ACB sections of $\B$ along $x(t)$. Then the Gram-Schmidt orthogonalization algorithm applied to $\{X_1,\hdots,X_s\}$ produces a $g$-orthonormal family of ACB sections of $\B$ along $x(t)$.  
\end{lemma}

\begin{proof}
Recall that the Gram-Schmidt algorithm maps a set $\{X_1,\hdots,X_s\}$ into a $g$-orthonormal set $\{X_1'',\hdots,X_s''\}$ constructed in the following way
\begin{align*}
&X_1\mapsto X_1':=X_1,\\
&X_2\mapsto X_2':=X_2-\rzut_{X_1}X_2,\\
&\hdots\\
&X_s\mapsto X_s':=X_s-\sum_{i=1}^{s-1}\rzut_{X_i}X_s,\\
&X_i'\mapsto X_i'':=\frac{1}{g(X_i',X_i')}X_i'\quad \text{for $i=1,\hdots,s$}; 
\end{align*}
where $\rzut_{X}Y:=\frac{g(X,Y)}{g(X,X)}X$ denotes the $g$-orthogonal projection of $Y$ on the space spanned by $X$.

Now it is enough to use the following elementary facts concerning ACB functions:
\begin{itemize} 
	\item A sum and a difference of two ACB functions is an ACB function. 
	\item A product of two ACB functions is an ACB function. 
	\item If an ACB function $\phi:[t_0,t_1]\ra \R$ is separated from zero, then $\frac 1\phi$ is an ACB function on $[t_0,t_1]$. 
\end{itemize}
Clearly, in every step of the Gram-Schmidt algorithm we apply one or more of these elementary operations to ACB sections (we use the fact that $g$ is ACB and that for every ACB non-zero vector $X$ the values of $g(X,X)$ are separated from zero on $[t_0,t_1]$.) Thus as a result we also obtain ACB sections $X_i''$'s. 
\end{proof}

\section*{Acknowledgements}
This research was supported by the National Science Center under the grant DEC-2011/02/A/ST1/00208 ``Solvability, chaos and control in quantum systems''.

\bibliographystyle{alpha}      
\bibliography{bibl}   

\end{document}